\documentclass[12pt,a4paper]{article}

\usepackage{amsmath,amsfonts,amssymb,amsthm,graphicx}     %formules mathématiques
\usepackage{color}

\newcommand{\R}{\mathbb{R}}

\newcommand{\E}{\mathbb{E}}
\newcommand{\Pb}{\mathbb{P}}

\newcommand{\ep}{\varepsilon}

\newtheorem{thm}{Theorem}%[section]
\newtheorem{rema}{Remark}

\begin{document}

\title{Ornstein-Uhlenbeck type processes\\ with heavy distribution tails}

\author{K.\ Borovkov\footnote{Department of Mathematics and Statistics, The University
of Melbourne, Parkville 3010, Australia. E-mail: borovkov@unimelb.edu.au}
 \ and\ G.\ Decrouez\footnote{Department of Mathematics and Statistics, The University of
Melbourne, Parkville 3010, Australia. E-mail: dgg@unimelb.edu.au}
 }

\maketitle

\begin{abstract}
We consider a transformed Ornstein-Uhlenbeck process model that can be a good candidate
for modelling real-life processes characterized by a combination of time-reverting
behaviour with heavy distribution tails. We begin with presenting the results of an
exploratory statistical analysis of the log prices of a major Australian public
company, demonstrating several key features typical of such time series. Motivated by
these findings, we suggest a simple transformed Ornstein-Uhlenbeck process model and
analyze its properties showing that the model is capable of replicating our empirical
findings. We also discuss three different estimators for the drift coefficient in the
underlying (unobservable) Ornstein-Uhlenbeck  process which is the key descriptor of
dependence in the process.


\end{abstract}

\section{Introduction}

There exist a large number of mathematical models designed to reproduce the
dynamics of financial data time series and aimed at capturing the key features
of their behaviour, one of the most important of them being heavy distribution
tails. A recent monographic reference presenting an up-to-date overview of the
field is~\cite{AnDaKrMi},  two further monographic references devoted
specifically to various aspects of using heavy-tailed distributions
being~\cite{Ra,AdFeTa}. Despite the abundance of (sometimes rather
sophisticated) such models, one cannot claim that there is single widely
accepted satisfactory model (or even class of models) for the behaviour of the
log prices of traded financial assets. Among the main reasons for that is the
ever-changing economical environment that must have substantial effect on the
dynamics of the financial time series and hence keeps making the already
verified and fitted models obsolete. Furthermore, different uses (which may
include options pricing and/or forecasting, at different time scales) of the
models stipulate different requirements on their structure and properties,
which means there can hardly be a single answer to all the requests. Hence the
continuing interest in considering alternative mathematical models that may be
successfully used to either get insights in the nature of the log price
dynamics or even solve some practical applied problems.

The present note may be considered as a follow-up to a relatively recent
paper~\cite{JePe} which analysed a class of rather simple diffusion models for the log
prices. The class consists of transformed Ornstein-Uhlenbeck (OU) processes of the form
$Y_t= h(X_t),$ where $h$ is a smooth enough strictly increasing function,
\begin{equation}
 \label{ou1}
dX_t = -\alpha X_t dt + \tau dW_t,  \qquad t\ge 0,
\end{equation}
$W_t$ is the standard Wiener process, and $\alpha, \tau >0$ are constants. Clearly,
both transition probabilities and stationary distribution are readily available for the
process $Y_t$, and it was shown in~\cite{JePe} that this class of models (i)~allows a
closed-form expression for the likelihood function of discrete time observations,
(ii)~allows the possibility of heavy-tailed observations and (iii)~allows the analysis
of the tails of the increments. The paper also discussed fitting the model to the log
share prices, via approximating by it the hyperbolic diffusion process considered
in~\cite{BiSo}, where the latter model was found  to possesses a number of properties
encountered in empirical studies of stock prices and was ``rather successfully fitted
to two different (Danish) stock price data sets".

However, despite one of the listed in~\cite{JePe} motivations for considering the class
being the need to be able to model heavy tailed data, the paper only goes as far as to
show (Proposition~2) how to choose $h$ to obtain tail behavior of the form
\[
\Pb (Y_t > x) = x^{c}\exp (- a x^\delta + b) (1 + O(x^{-1})), \qquad x\to \infty,
\]
with $\delta \in [1,2],$ $a >0$ and $b,c\in \R$, while Theorem~3 of the paper gives
conditions on $h$ under which the transition density will have exponential decay at
infinity. This kind of tails can hardly be classified as being heavy. In fact, as it is
well-known and will be confirmed in our Section~\ref{Sec2} below,  one often deals with
a power function tail decay for real-life data sets.

The main objectives of the present paper include eliciting conditions on $h$ under
which the distribution tails (for the stationary law of $h(X_t)$, the distributions of
the increments $h(X_t)-h(X_s)$ and the transition probabilities) for the process $Y_t=
h(X_t)$ will have regularly varying behaviour at infinity thus agreeing with empirical
observations, and showing that the dependence structure of such a model is also
consistent with a typical Australian stock log price behaviour. We used BHP Billiton
Ltd prices from 2003--2010 for illustration purposes, but obtained similar results when
doing our calculations for several other major companies, both from mining and other
sectors of the economy as well, that are listed on the Australian Securities Exchange
(ASX), including the ANZ Banking Group Ltd and TOLL Holdings Ltd.

Further arguments in favour of the proposed simple model include the fact that there is
good Gaussian fit for the empirical copulas of the process' values, and the observation
that the parameter $\alpha$ characterising the dependence structure of the process can
be easily estimated using  rank correlations techniques, and that these estimates
computed from the time series increments for different time lags agree with each other.

One of the important aspects of mathematical modelling of time series is the choice of
the time scale which the modelling will be aimed at. Firstly, the dynamics at different
scales can be different, as the key driving factors at these scales may be
different\,---\,and changing because of changing environment, as it may be happening
with the small scale picture over the last years, due to the wide spread of fast
algorithmic trading. Secondly, capturing the detail of the small-scale behaviour of
prices can be of little help if one is interested in medium or long-term modelling. In
our study, we are mostly interested in medium-term modelling, for daily data with time
horizons of the order of~$10-50$ days. The empirical data we used seem to confirm the
appropriateness of the suggested model: even though some of its parameters may be
changing with time, it appears that its structure remains consistent with observations.

For further reading and references concerning using OU processes for modelling
of financial time series, the interested reader is referred to~\cite{MaMuSz}.

The paper is organised as follows. In Section~\ref{Sec2} we present the results of an
exploratory analysis of the BHP Billiton Ltd log prices time series, to elicit the key
features that successful candidate models (for time horizons around~10 days) must
possess (tail behaviour, dependence structure). Section~\ref{Sec3} presents a
theoretical analysis of the model $Y_t= h(X_t)$, $\{X_t\}$ following~\eqref{ou1},
demonstrating conditions on $h$ that lead to regularly varying distribution tail
behaviour and also showing that the process' dependence structure (in terms of rank
correlations) is consistent with the real life data properties discussed in
Section~\ref{Sec2}. Section~\ref{sec_com} contains a number of comments (mostly
concerned with the estimation of~$\alpha$) on fitting the suggested model to data.
Section~\ref{Sec_proofs} presents the proofs of the theorems formulated in the paper.

\section{Distributional properties of log returns}
\label{Sec2}

A lot of work has been done on studying the distributional properties of the stock
(log) price time series, and we refer the interested reader
to~\cite{AnDaKrMi,AdFeTa,Ra,HeLe} and the numerous references therein for illuminating
discussions of the past findings. However, for the purposes of this research, we
decided to do new statistical analysis of a single stock price time series, with a view
of looking at the aspects of the empirical data set that can help us clarify the
suitability of the transformed OU process model. In this section, we use the time
series of daily (closing) prices $S_t$ for BHP Billiton Ltd stock as quoted by the ASX,
$t=0, 1,\ldots, N=2061,$ the day $t=0$ being 1~January 2003, $t=N$ corresponding to
17~December 2010, to elicit the typical distributional properties of the log prices
$Z_t:= \ln S_t$. As we pointed out earlier, we did the  same calculations for the stock
prices of several other major Australian companies from different sectors of economy
(e.g.\ ANZ Banking Group Ltd and TOLL Holdings Ltd), which yielded quite similar
results, so it appears that the data used in this section may be deemed to be typical.

All the data analysis and graphics were done in {\sc Matlab.}

\subsection{The tail behaviour}
\label{Sec2_1}

The usual starting point in data analysis of this kind is to make normal
Q-Q-plots. Since we are interested in a dynamic model for the data, we do that
for the increments
\[
\Delta_\delta Z_t:=Z_t-Z_{t-\delta},\qquad t=\delta,\delta+1,\ldots, N,
\]
for different values of $\delta$. Clearly, $\Delta_\delta Z_t$ is the log
return of the stock over a period of time of length $\delta$ ending at time
$t$. A typical result is shown in Fig.~\ref{fig_qq} displaying the plot for
$\delta=20$: there is a very good straight line fit in the central part
(representing more than 95\% of all data, so that the bulk ``middle part" of
the distribution looks pretty much normal), and then in the end regions there
are clear deviations from the straight  line indicating the presence of heavy
tails. Different time lag values in the range up to $\delta =100$ return rather
similar pictures.

\begin{figure}[ht] % OK
\centering
%\begin{tabular}{c}
   \includegraphics[width=8cm]{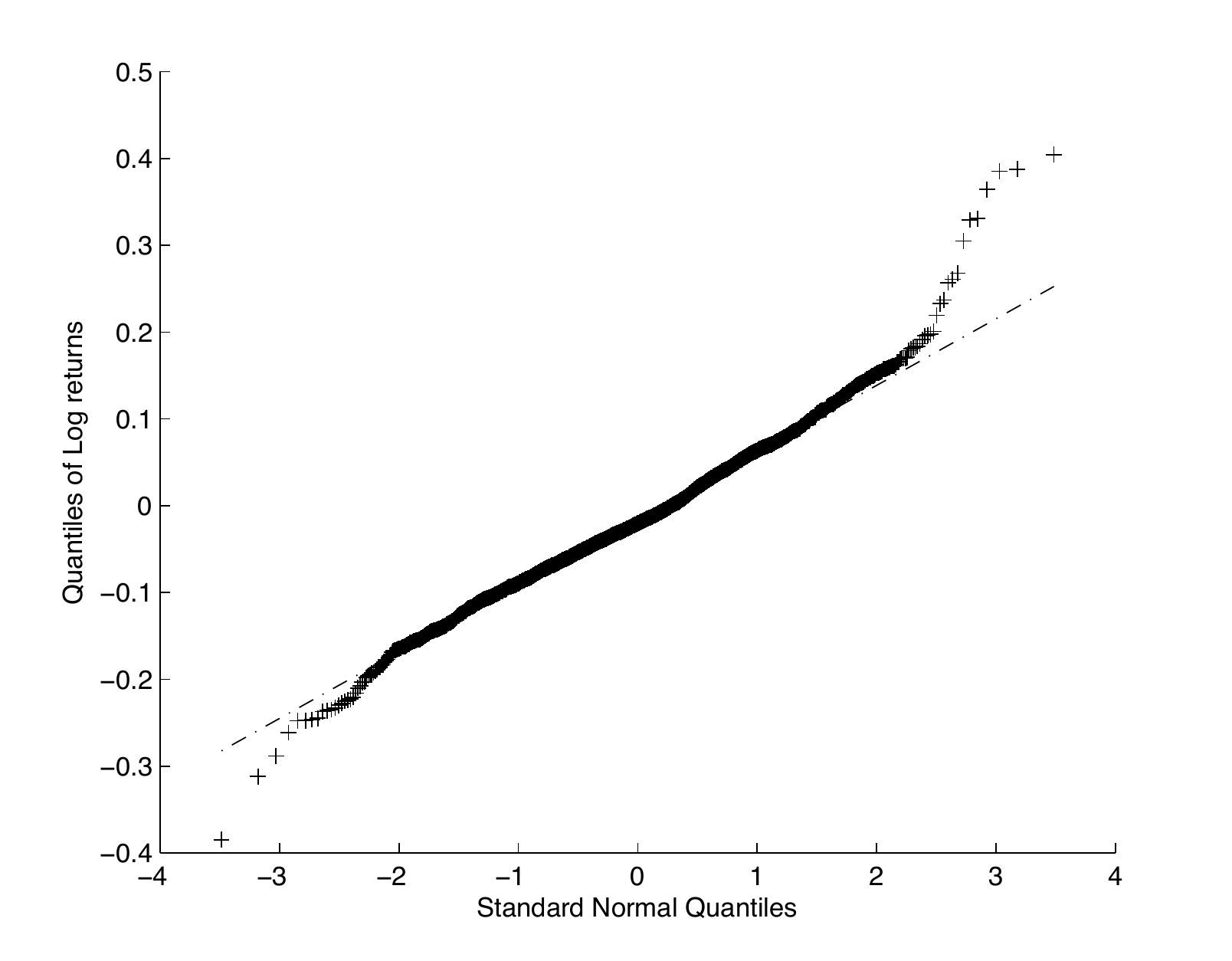}
%   \includegraphics[width=7.5cm]{motivation_right_tail}\\
%\end{tabular}
 \caption{\small The normal Q-Q plot for the 20 days log-returns $\{\Delta_{20} Z_t\}$.}
\label{fig_qq}
\end{figure}

%\begin{verbatim}motiv_qq.pdf\end{verbatim}: QQ normal plot for $\Delta_{20} Z_t$.

To elicit the character of the tail behaviour of the empirical distributions, one can
use the log-log plots of their tails. Figure~\ref{fig:lr} displays the natural
logarithms of the empirical distributions' left and right tails for the log returns
$\{\Delta_\delta Z_t\}$ against the logs of the data values, for four different time
lag values:  $\delta = 1,5,10$ and 20. The key common features of the plots are: a nice
parabolic shape in the ``middle region" (which indicates a good normal fit in that
region), and then a very good straight line fit for the ``remote regions" (comprising
about 5\% of all data) in all the cases, indicating the power decay of the
distributions. An interesting observation is that, for each of the right tails, the
slope of the fitting straight line (and hence the index of the approximating power
function) is one and the same, for all time lag values.

\begin{figure}[ht] % OK
\centering
\begin{tabular}{c}
   \includegraphics[width=7cm]{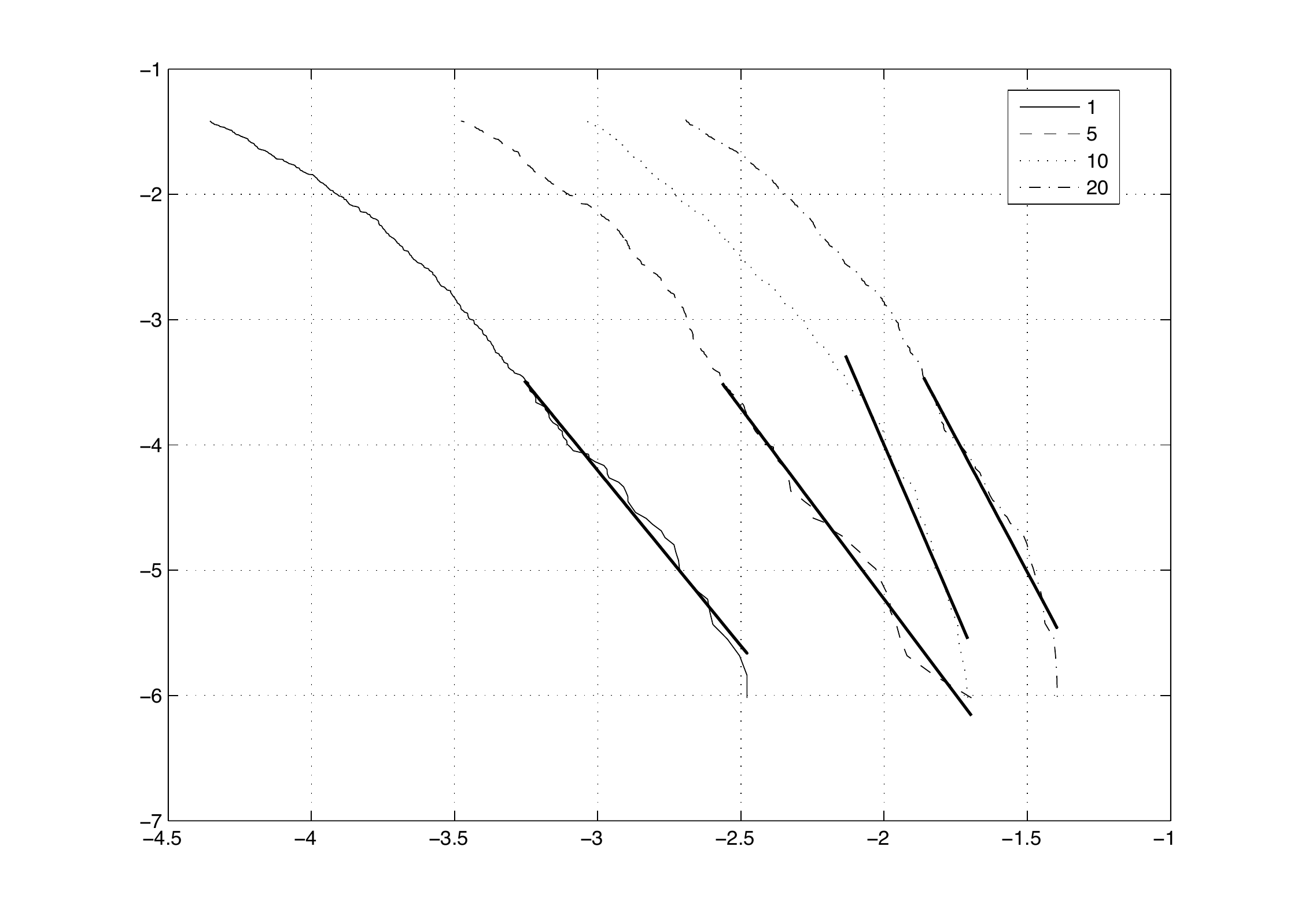}
   \includegraphics[width=7cm]{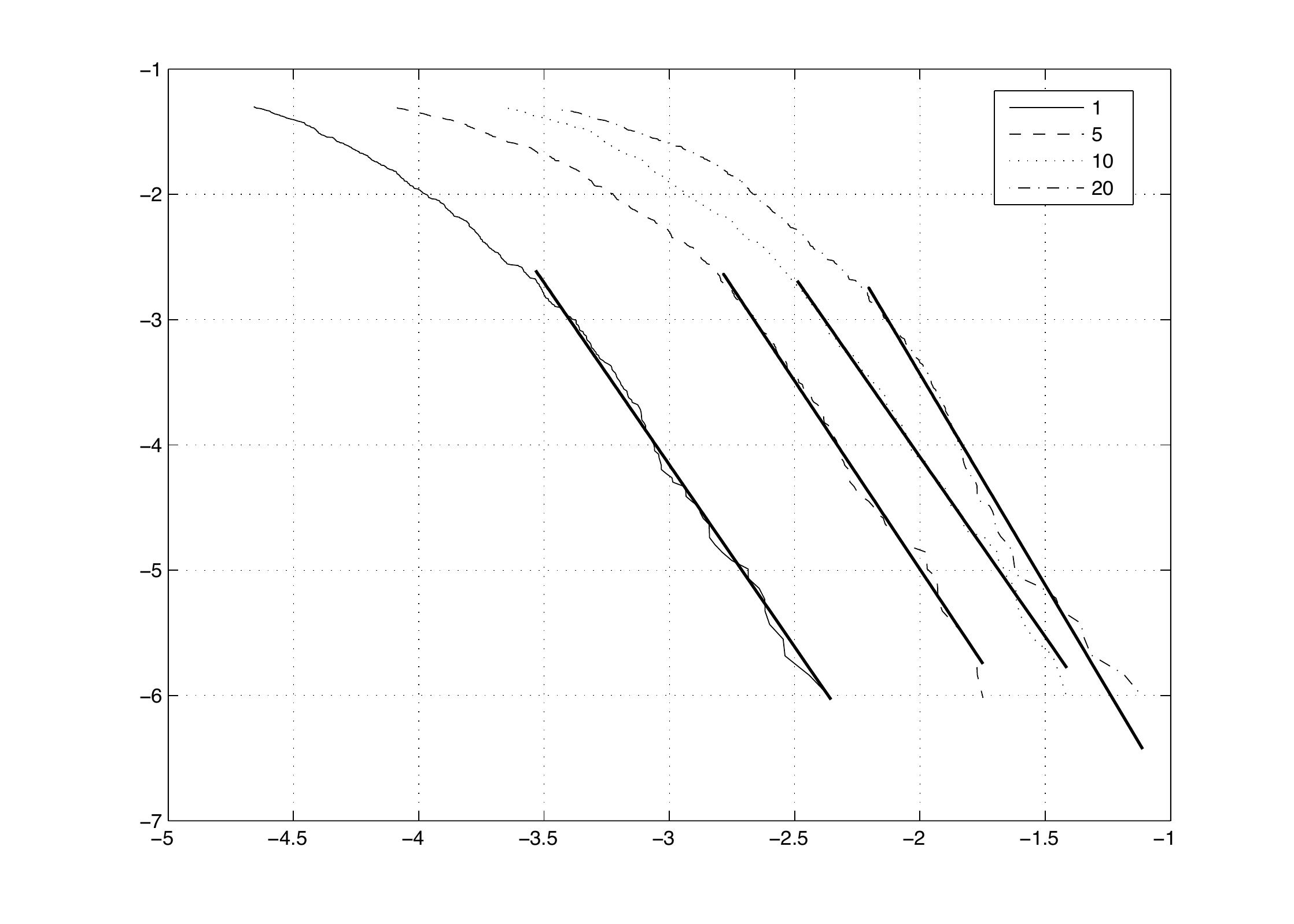}\\
\end{tabular}
 \caption{\small The log-log plots of the empirical distributions' tails of the BHP stock price log returns for time periods
 of $\delta= 1,5,10, 20$ (days).
 The left (right) pane presents the plots for the left (right, resp.) tails. Straight lines are fitted in each
 case using the mean squares (the total number of observations from these regions comprise about 5\% of all data).}
\label{fig:lr}
\end{figure}

The numerical (absolute) values $\tilde\beta_\pm (\delta)$ of the slope coefficients of
the fitted straight lines for different values of~$\delta$ are given in
Table~\ref{tab_1} (the subscripts $+/-$ correspond to the right/left tails). The table
also shows the values of the Hill estimators $\hat \beta_{\pm,\{p\}} (\delta)$ for the
exponents of the power functions specifying the tail decay rate~\cite{Hill}. Recall
that, under the assumption that the right tail of the theoretical distribution of the
data $V_1, \ldots, V_n$ with order statistics $V_{1,n}\ge \cdots\ge V_{n,n}$ has a
regularly varying form $x^{-\beta_+} L(x),$ $L(x)$ being a slowly varying function as
$x\to\infty$, the Hill estimator $\hat \beta_{+,\{p\}} $ of $\beta_+$ is computed
according to
\begin{equation*}%\label{hill}
\frac1{\hat{\beta}_{+,\{p\}}}  = \frac{1}{k-1}\sum\limits_{i=1}^{k-1}\ln \frac{V_{i,n}}{V_{k,n}}, \qquad p=100\,\frac{k}{n} .
\end{equation*}
Note that the estimator was shown to be consistent (as $n\to\infty$, $p \to 0$)
for observations forming a strongly mixing sequence~\cite{Root90}.

Observe that the estimates obtained by both methods are generally in good agreement,
especially for the right tail  which appears to be ``heavier". Note that for that tail
the values of the estimates for the regular variation index remain roughly the same for
all $\delta$ values considered.

\begin{table}[h]
\centering
\begin{tabular}{|c|c|c|c|c|c|}
\hline
    &  $\delta$ & 1 & 5 & 10 & 20 \\
     \hline
Estimators  & $\tilde\beta_- (\delta) \vphantom{\displaystyle\sum}$   & 2.75 & 3.28 & 5.29 & 4.29 \\
for the & $\hat\beta_{-,\{2\}} (\delta) \vphantom{\displaystyle\sum}$   & 2.88 & 2.91 & 5.93 & 4.09 \\
left tail: & $\hat\beta_{-,\{5\}} (\delta) \vphantom{\displaystyle\sum}$   & 2.80 & 3.42 & 3.78 & 4.53 \\
    \hline
Estimators   &  $\tilde\beta_+ (\delta) \vphantom{\displaystyle\sum}$   & 2.91 & 3.00 & 2.87 & 3.37 \\
for the  &$\hat\beta_{+,\{2\}} (\delta) \vphantom{\displaystyle\sum}$   & 3.05 & 3.47 & 3.13 & 3.55 \\
right tail: & $\hat\beta_{+,\{5\}} (\delta) \vphantom{\displaystyle\sum}$   & 2.94 & 3.11 & 2.85 & 3.24 \\
\hline
\end{tabular}
 \caption{\small Estimates for the exponents of the power functions describing the empirical distribution tail decay for
 log returns over periods of different lengths~$\delta$.
 Here $\tilde \beta_\pm$ are obtained as the slope coefficients of the fitted straight lines in Fig.~\ref{fig:lr}, while
 $\hat\beta_{\pm,\{p\}}$ are Hill estimators based on the left-most/right-most $p$\% of the data.}
\label{tab_1}
\end{table}

\subsection{Rank correlations}
\label{rancor}

Now we will turn to analysing the dependence structure of our time series. The
standard approach for such a task is based on working with linear correlation
coefficients. However, as we saw in Section~\ref{Sec2_1}, we are dealing here
with heavy-tailed data for which using linear correlations may be unwise. For
that reason, and also because our intention is to use for modelling purposes a
transformation of a Gaussian process with a simple dependence structure, we
will prefer to analyse rank correlation coefficients in this section (as they
are invariant under strictly increasing transformations and thus may be useful
for doing statistical inference for the underlying unobservable process).

As is well known, the values of the two most popular rank correlation measures,
Kendall's tau and Spearman's rho, are quite close to each other (which was also
confirmed by the data we deal with in this section). Because computing the
latter requires less computational effort, we chose to work with Spearman's
rho.

Recall (see e.g.\ \cite{McNeil}, pp.\,207, 229) that the theoretical value of
Spearman's rank correlation coefficient $\rho_S$ for a pair of random variables
$U,V$ is defined by
\begin{equation}
 \label{rhoS}
\rho_S(U,V ) = \rho\left(F_U(U), F_V(V )\right),
\end{equation}
where $F_U$ and
$F_V$ are the (marginal) distribution functions of $U$ and $V$, resp., and
$\rho$ is the linear correlation. The standard estimator of $\rho_S(U,V )$ from
a sample $(U_j,V_j),$ $j=1,\ldots,n,$ is calculated using the ranks of the
variables within the respective univariate samples as
\begin{equation}
 \label{rhoSest}
\hat{\rho}_S :=\frac{12}{n(n^2-1)}\sum\limits_{j=1}^{n} \left(\textrm{rank}(U_j)-\frac{{n+1}}{2} \right)\left(\textrm{rank}(V_{j })-\frac{n+1}{2}  \right).
\end{equation}

It is clear that calculating rank correlations for the whole sample of daily
data over the eight year long period, without careful detrending and possibly
some further data pre-processing, would hardly be meaningful. Furthermore, as
we pointed out earlier, we would be most interested in modelling the price
processes over medium-term time intervals (about 100 days long), over which one
can expect data to follow a reasonably stationary process. And indeed,
computing correlations for data from time intervals of such lengths leads to
rather stable results. A typical representative is depicted in
Fig.~\ref{fig_logprice_corr}, showing Spearman's rho values for the data
consisting of the pairs $(Z_t, Z_{t+k})$, $t=0,1,\ldots, 200-k,$ for values
$k=1,\ldots, 35$. One can observe that the plotted line looks pretty much like
an exponential curve, of the form $e^{-ck}$, a behaviour typical of
autoregressive processes of the first order (or  OU processes in continuous
time).

\begin{figure}[ht]  % OK
\centering
   \includegraphics[width=8cm]{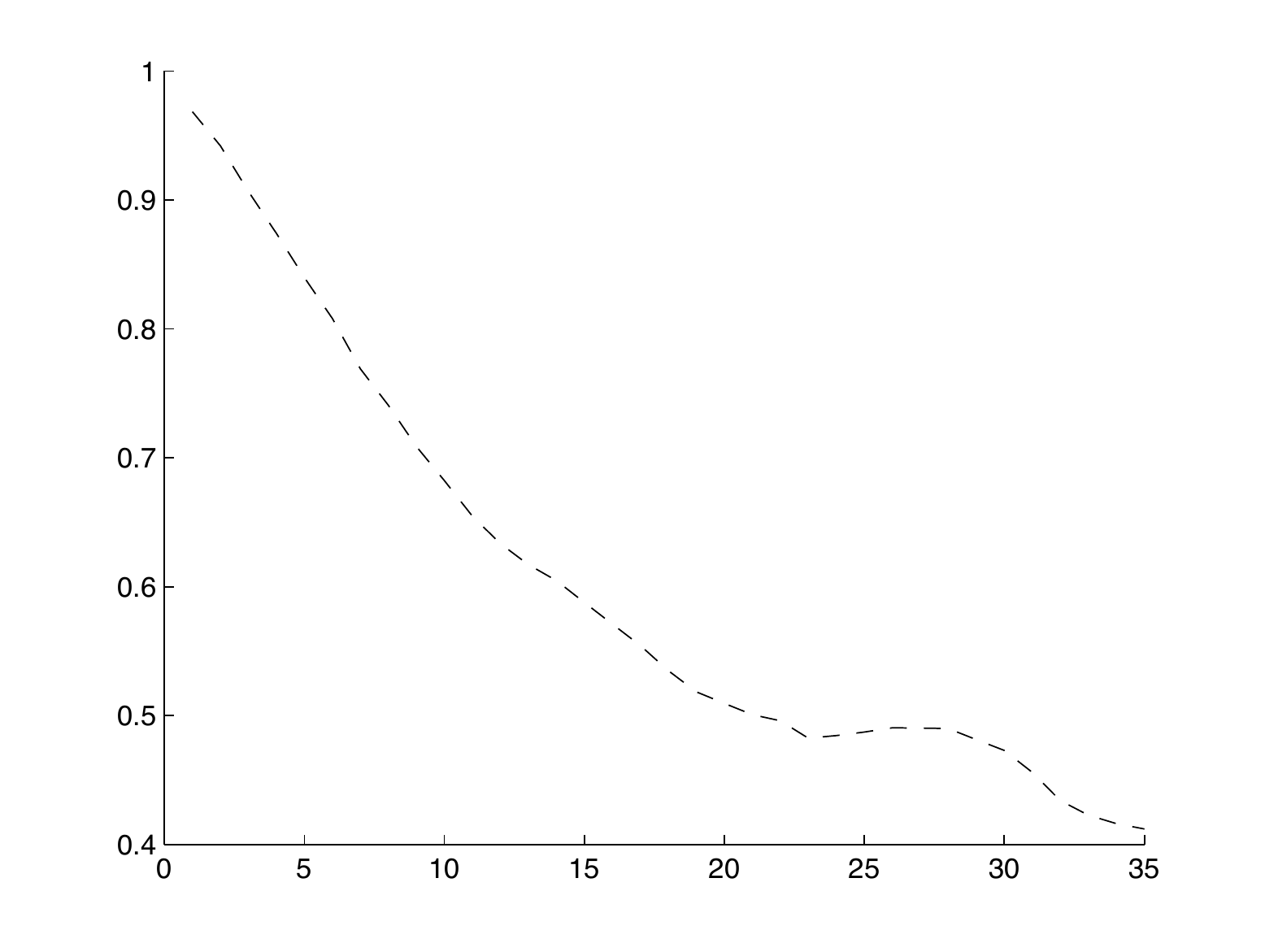}
 \caption{\small Spearman's rho for $(Z_t, Z_{t+k})$, $k=1,\ldots, 35$, from the first 200 data points.}
\label{fig_logprice_corr}
\end{figure}

For the {\em increments\/} of the time series values, i.e.\ the log returns
over time periods of a fixed length~$\delta$, the picture is even more
interesting. Following our general approach, here we will pool data from the
whole eight year long period of observations, as the influence of the possible
(slow) trend present in the time series of the values of the increments of
$Z_t$ over relatively short lags would be quite small. Fig.~\ref{fig_incr_corr}
shows three plots of Spearman's rho's values calculated for the data consisting
of the pairs $(\Delta_\delta Z_t, \Delta_\delta Z_{t+k}),$ $t=\delta, \delta
+1,\ldots, N-k,$ that are plotted for $k=1,2,\ldots, 100,$ for the values of
$\delta$ equal to 20, 40 and~60.

\begin{figure}[ht] % OK
\centering
   \includegraphics[width=8cm]{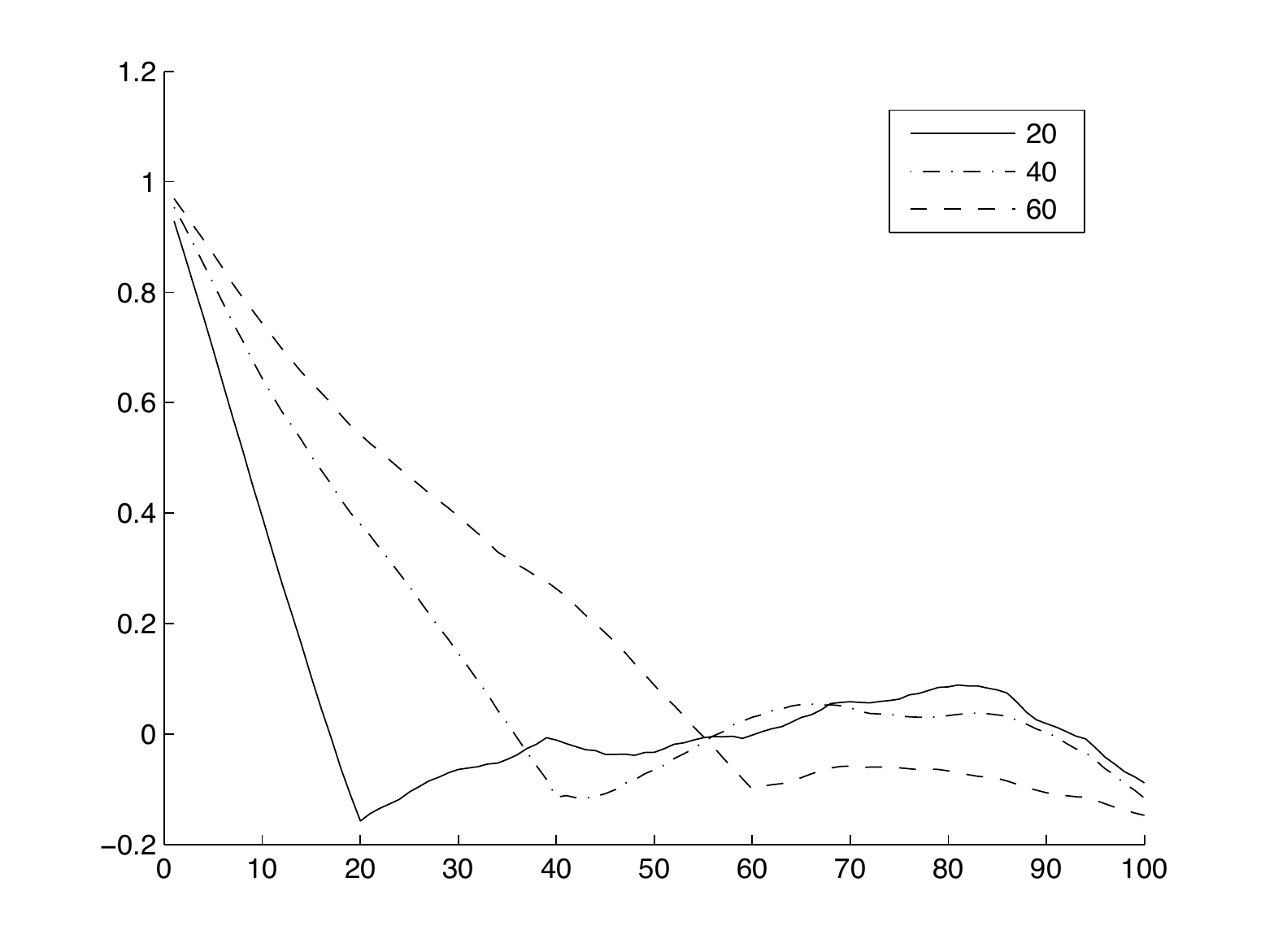}
 \caption{\small The plots of Spearman's rhos for the pairs of log returns $(\Delta_\delta Z_t, \Delta_\delta Z_{t+k}),$ $t=\delta, \delta
+1,\ldots, N-k,$  over time periods of $\delta=20, 40$ and~60 days, plotted for $k=1,2,\ldots, 100.$}
\label{fig_incr_corr}
\end{figure}

A remarkable common feature of these plots is that they decay in an almost
linear way for lag values $k$ from 1 to $\delta$, attaining (small) negative
values at the minimum points, and then start slowly growing.

\subsection{Copulas}

To further analyse the dependence structure of our data, we turn to copulas.
One standard way of graphical representation of empirical copulas is to make
scatterplots of the pseudo-sample obtained by transforming the components of
the sample points using the marginal empirical distribution functions (see
e.g.\ p.\,232 in~\cite{McNeil}).

We will begin with considering the original data set~$\{Z_t\}$. To avoid the
interference of the possible long-term trend, first we process the original
data from blocks of length~100. We consider $20$ subsamples
\[
\{(U_{k,j}, V_{k,j}):=(Z_{t_k+j}, Z_{t_k+j+5}) , \ j=0,1,\ldots, 94\}, \quad\mbox{\rm
where $t_k= 100k,$ $k=0,1,\ldots, 19$},
\]
and for each of them produce a sub-pseudosample given by
\[
(\hat{U}_{k,j},\hat{V}_{k,j}) = \left( \frac{\textrm{rank}(U_{k,j})}{96}
,\frac{\textrm{rank}(V_{k,j})}{96}   \right), \qquad j=0,\hdots,94,
\]
where $\textrm{rank}(U_{k,j})$ is the rank of $U_{k,j}$ in the $k$th subsample
$U_{k,0}, U_{k,1}, \ldots, U_{k,94}$, and likewise for $\textrm{rank}(V_{k,j})$. The
aggregate pseudosample is then obtained as the union of these 20 sub-pseudosamples and
is depicted on the left pane in Fig.~\ref{fig:copulas_1}.

\begin{figure}[ht]  %
\centering
\begin{tabular}{c}
   \includegraphics[width=7cm]{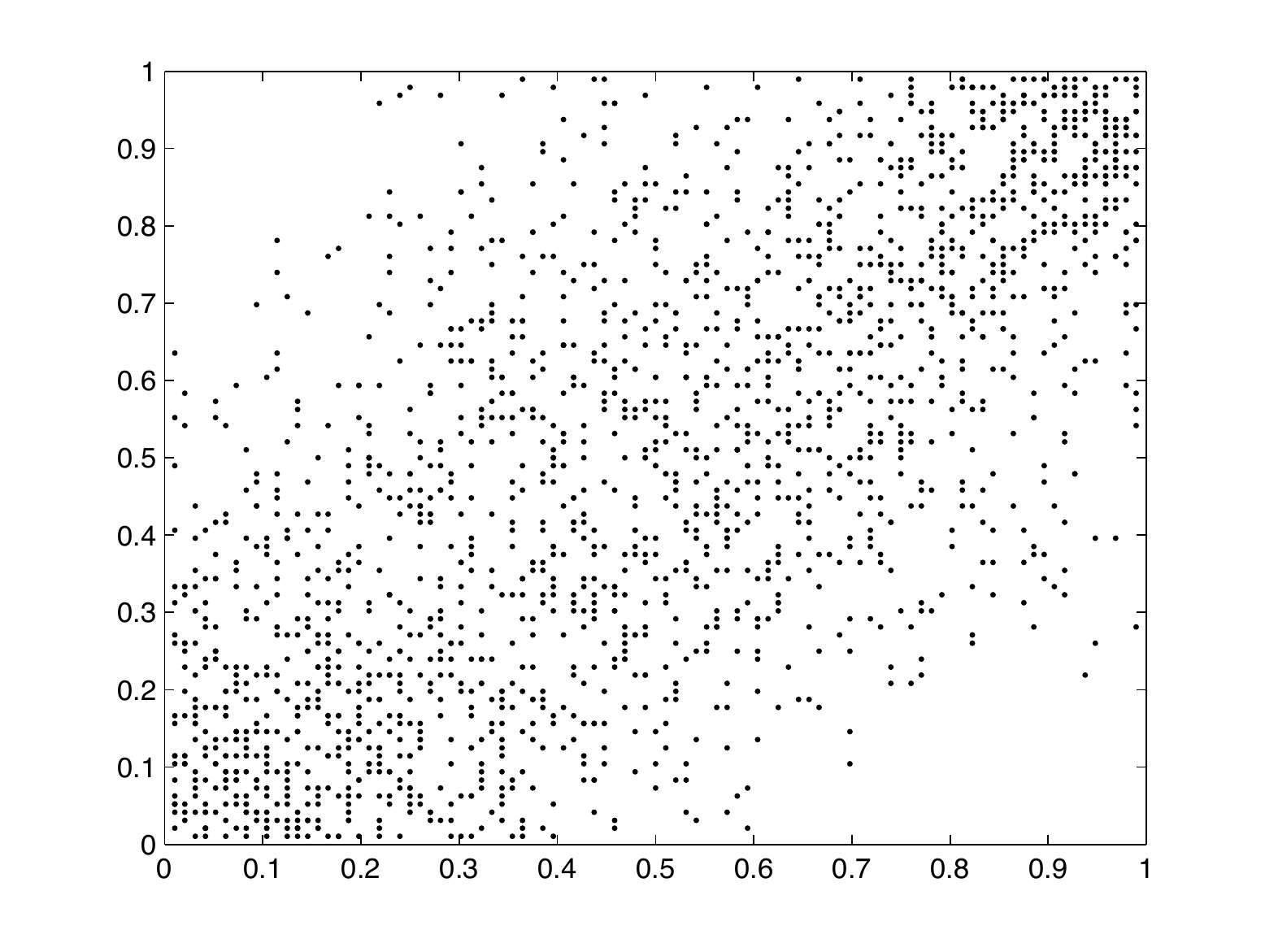}
   \includegraphics[width=7cm]{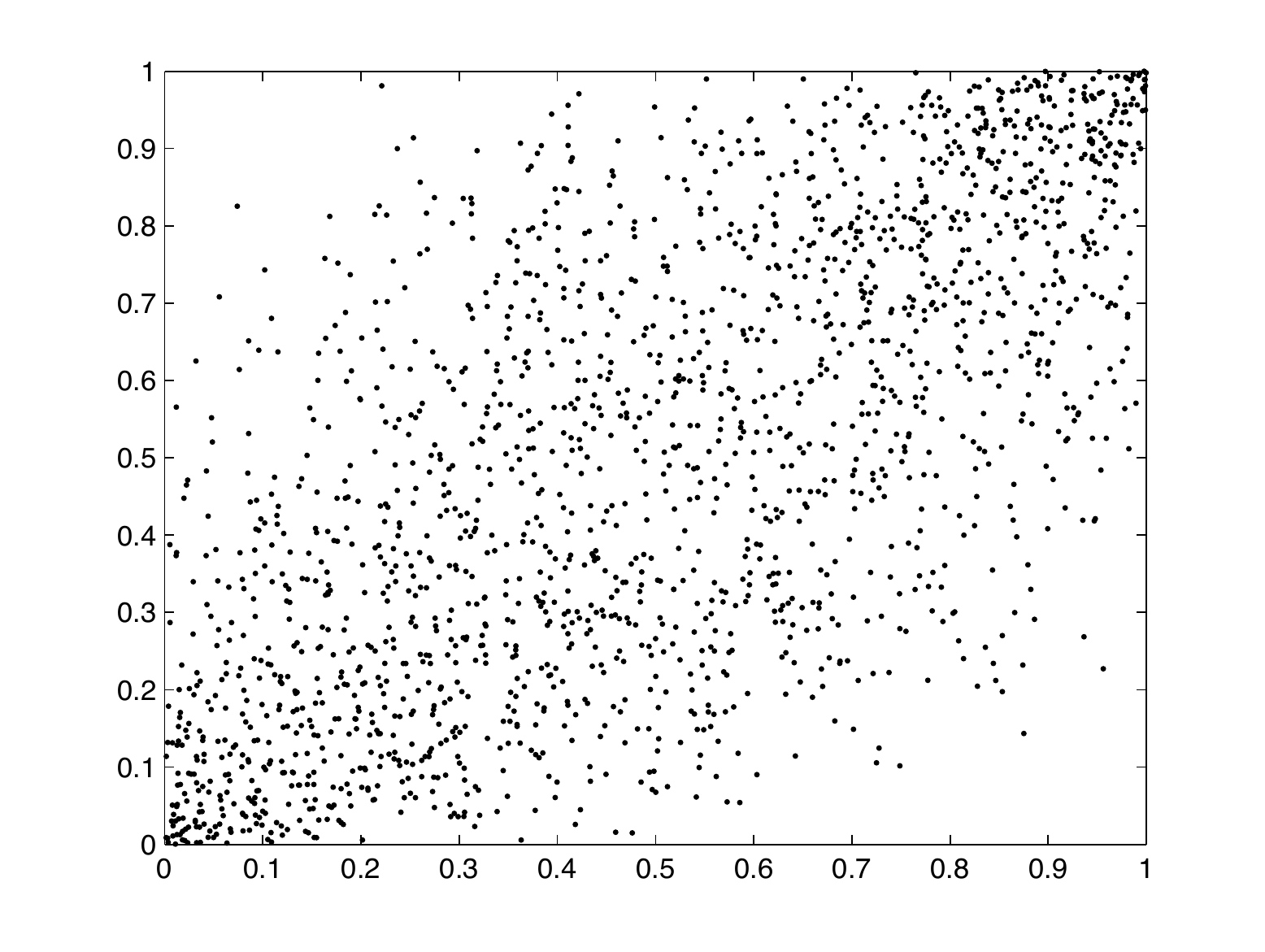}\\
\end{tabular}
 \caption{\small The aggregate scatterplot for the pseudosamples obtained from the
 data $(Z_t, Z_{t+5})$ (left pane) and the scatterplot of $N$ simulated points from
 the Gauss copula $C_\rho^{\rm  Ga}$ with $\rho =0.693$ obtained from the estimate of Spearman's
 rho.}
\label{fig:copulas_1}
\end{figure}

Next to that scatterplot representing the dependence structure for the log
prices at lag five, we put the scatterplot of $N$ points simulated from the
Gauss copula
\begin{equation}
 \label{Ga_co}
C_\rho^{\rm Ga}(u_1,u_2)=\int_{-\infty}^{\Phi^{-1}(u_1)}\int_{-\infty}^{\Phi^{-1}(u_2)}
 \exp\left\{-\frac{x_1^2-2\rho x_1x_2+x_2^2}{2(1-\rho^2)}\right\} \frac{ dx_1dx_2}{2\pi(1-\rho^2)^{1/2}},
\end{equation}
$\Phi$ being the standard normal distribution function and $\Phi^{-1}$ its inverse,
with the parameter $\rho$ value obtained from our estimate for Spearman's rho. Recall
that, for any bivariate random vector $(X_1, X_2)$ following a  meta-Gaussian
distribution with copula~\eqref{Ga_co}, one has
\begin{equation}
 \label{rhorho}
\rho_S (X_1, X_2) = \frac{6}{\pi} \arcsin \frac{\rho}2
\end{equation}
(see e.g.\ Theorem~5.36 in~\cite{McNeil}), so that we can use the estimate $\hat\rho=
2\sin (\pi\hat\rho_S/6)$ for~$\rho$. In the case of the data set $(Z_t, Z_{t+5})$, this
leads to $\hat\rho= 0.693.$

No localisation of data followed by the subsequent  aggregation is needed for the
increments of the original time series values. Figure~\ref{fig:copulas_2} presents the
scatterplot of the pseudosample constructed from the data set $\{(\Delta_{20} Z_t,
\Delta_{20} Z_{t+5}),\, t=20,21, \ldots, N-5\}$, alongside with the scatterplot of a
sample simulated from the Gauss copula with the parameter value $\rho= 0.695$ (also
estimated via~\eqref{rhorho} from the respective Spearman's rho).

\begin{figure}[h]
\centering
\begin{tabular}{c}
   \includegraphics[width=7cm]{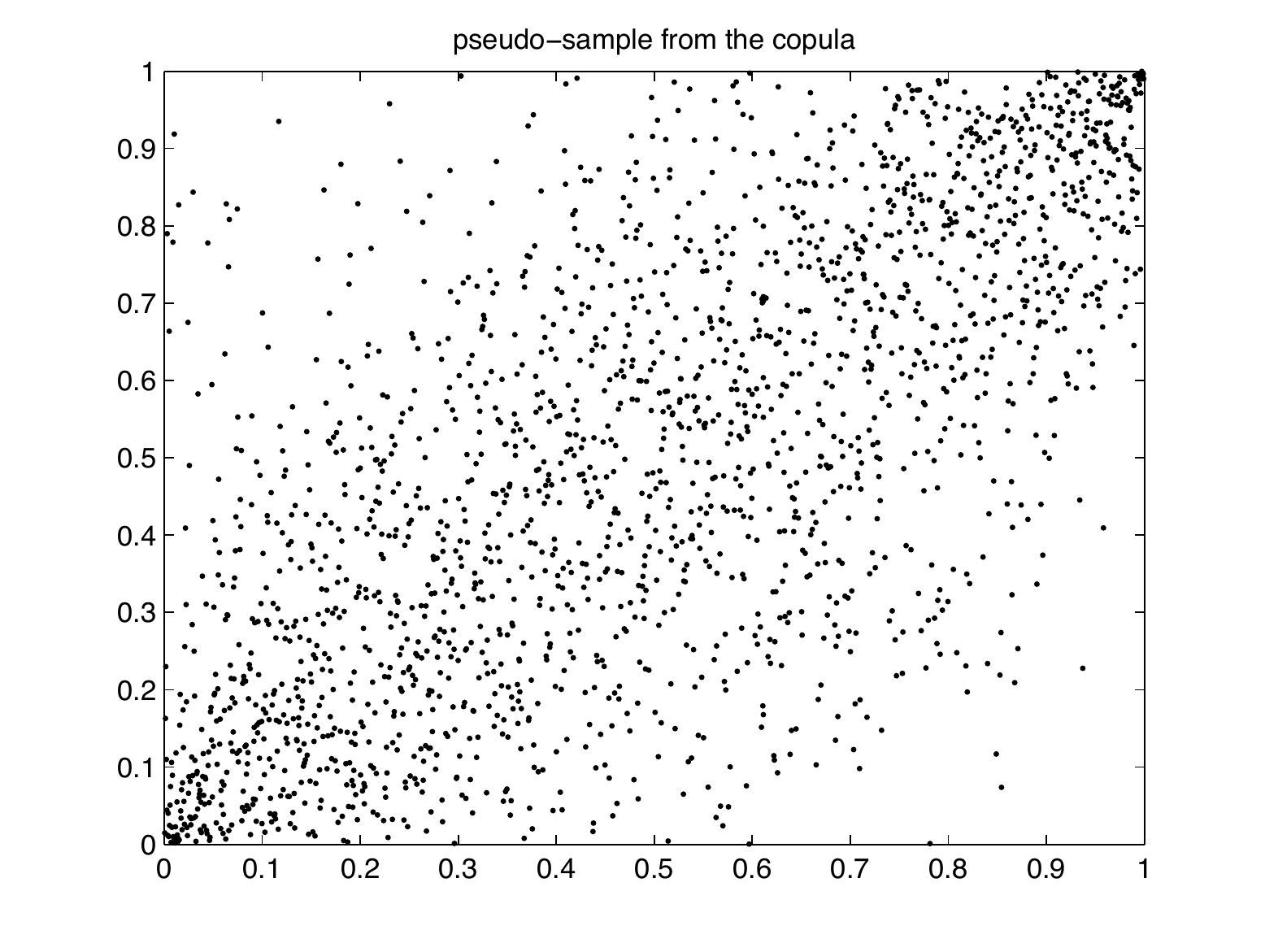}
   \includegraphics[width=7cm]{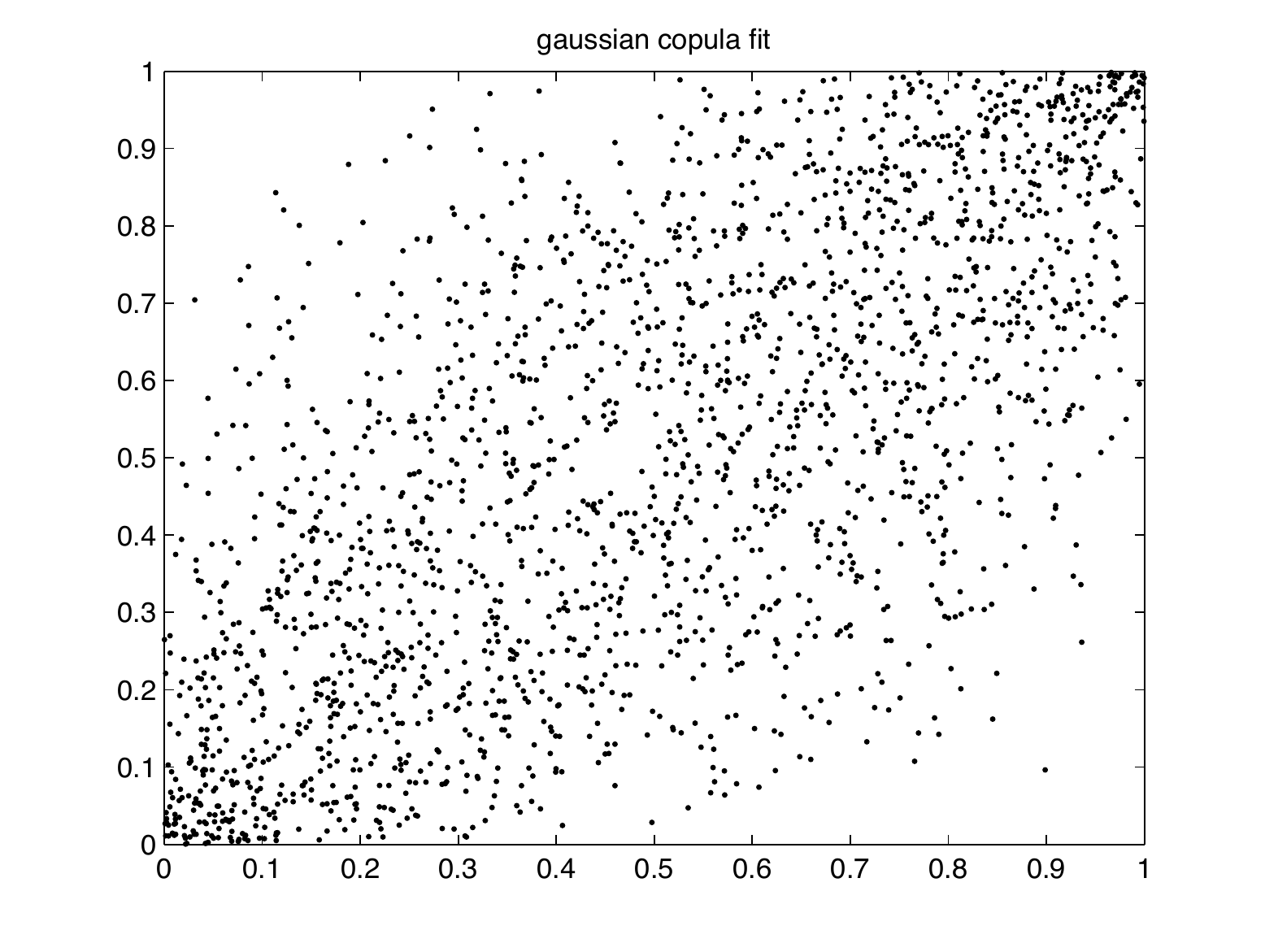}\\
\end{tabular}
 \caption{\small The scatterplot for the pseudosample obtained from the data $(\Delta_{20} Z_t, \Delta_{20} Z_{t+5})$ (left pane) and the scatterplot of $N$ simulated points from
 the Gauss copula $C_\rho^{\rm  Ga}$ with $\rho =0.695$ obtained from the estimate of Spearman's rho.
 }
\label{fig:copulas_2}
\end{figure}

We can see that, in both cases, there is a reasonable agreement of the empirical data
with the fitted Gauss copulas. To confirm this observation numerically, we used a
recently suggested method for calculating approximate $P$-values for testing the
goodness-of-fit by parametric copula families~\cite{KoYa1,KoYa2}. As the method is
valid for i.i.d.\ samples, to use it in our situation we first had to weaken the
dependence between sample points by ``rarefying" the data. We applied the test to
samples obtained by taking each 10th pair of log-returns over two time intervals of
length $\delta=1, 5, 10$ and $20$ days, one shifted by 5~days relative to the other,
thus dealing with four bivariate samples each consisting of 206 points of the form
$(\Delta_\delta Z_t, \Delta_\delta Z_{t+5})$, $t= 10,20,30,\ldots, 2060.$ The method
uses Monte Carlo techniques, and based on 100 simulations from the respective Gauss
copulas for each of the $\delta$ values considered, we obtained the following
$P$-values: 0.675 ($\delta=1$); 0.332 ($\delta=5$); 0.214 ($\delta=10$); 0.433
($\delta=20$);  so that the Gauss copula hypothesis was not to be rejected.

\medskip

In conclusion of this section, we   summarise its key findings that will be
referred to in Section~\ref{Sec3}:
\begin{enumerate}
\item[{\bf [F$_1$]}] The log returns $\Delta_\delta Z_t$ over periods of
    $\delta$ days have distributions of which the ``central parts" (about
    95\% of all data) are very well fitted by the respective parts of
    normal distributions.

\item[{\bf [F$_2$]}] The empirical distributions of the log returns
    $\Delta_\delta Z_t$ have power tails, with the exponent of the power
    function (roughly) independent of the lag~$\delta$.

\item[{\bf [F$_3$]}] Spearman's rho for $(Z_t, Z_{t+k})$ decays as an
    exponential function of~$k$.

\item[{\bf [F$_4$]}] As a function of $k$, Spearman's rho for
    $(\Delta_\delta Z_t, \Delta_\delta Z_{t+k})$ behaves in the fashion
    presented in Fig.~\ref{fig_incr_corr}.

\item[{\bf [F$_5$]}] For fixed $k$ and $\delta$, the empirical copulas for
    both $(Z_t, Z_{t+k})$ and $(\Delta_\delta Z_t, \Delta_\delta Z_{t+k})$
    agree with the respective Gauss copulas $C_\rho^{\rm Ga}$, with
    the parameter $\rho$ estimated from the calculated values of Spearman's
    rho for the samples. We illustrated that for $k=5$, $\delta=20$, and
    gave approximate $P$-values for a goodness-of-fit test for Gauss
    copulas for $(\Delta_\delta Z_t, \Delta_\delta Z_{t+k})$ with $k=5$ and
    $\delta=1,5, 10$ and 20, but the situation is similar for other
    values of the quantities as well.

\end{enumerate}

\section{The model and its key properties}
\label{Sec3}

In this section, we present a formal description of our simple model suggested by the
findings {\bf [F$_1$]}--{\bf [F$_5$]} and demonstrate that it does have properties
consistent with these empirical facts.

Let $\{X_t\}$ be a stationary  OU process driven by the stochastic differential
equation~\eqref{ou1}. It is well known that $\{X_t\}$ is a Gaussian process,
with $X_t\sim N(0,\sigma^2)$, where $\sigma^2=\tau^2/(2\alpha)$. Furthermore,
for $0\le s\le  t$, the conditional distribution of $X_t$ given $X_s=x$ is
normal with mean and variance given by $ m_{t-s} (x)$ and $\sigma^2_{t-s}$,
respectively, where we used notation
\[
 m_{u} (z):= z e^{-\alpha u}
, \qquad
 \sigma^2_{u}:=\frac{\tau^2}{2\alpha}\, (1-e^{ -2\alpha u }),
\]
so that the (linear) correlation function of the process has the form
\begin{equation}
 \label{ou_corr}
 \rho (X_s, X_t)=e^{-\alpha |t-s|}.
\end{equation}

For a    strictly increasing continuous mapping $h:\R\to\R$, consider the process
$$
Y_t := h(X_t), \qquad t\ge 0.
$$
It is obvious that $\{Y_t\}$ is also a stationary Markov process and, moreover, that if
$h$ is twice continuously differentiable then $\{Y_t\}$  will be a diffusion process as
well. Observe that all finite-dimensional distributions of $\{Y_t\}$ are meta-Gaussian:
for any $0\le t_1 < t_2 <\cdots < t_n$, the vector $(Y_{t_1}, \ldots, Y_{t_n})$ has a
Gauss copula with the correlation matrix $\left(\rho_{ij}=e^{-\alpha |t_i -
t_j|}\right)_{i,j\le n}$ and identical univariate marginal d.f.'s all equal to
$F(y):=\Pb (Y_t \le y)= \Phi (h^{-1}(y)/\sigma)$, $\Phi$ being the standard normal~d.f.

Our first objective will be to determine conditions on $h$ under which one or both of
the distribution tails of $Y_t$ is/are of regular variation, so  that for the d.f.\ $F$
of $Y_t$ one has
\begin{align}
 \label{rt}
\overline{F} (y)  := 1 &-F(y)  = y^{-\beta_+} L_+(y)
\\
&\mbox{and/or} \notag
 \\
 F(-y)& = y^{-\beta_-} L_-(y),
 \label{lt}
\end{align}
where $\beta_\pm>0$ and $L_\pm$ are functions slowly varying  as $y\to\infty.$

To this end, we will introduce two conditions {\bf [A$_+$]} and {\bf [A$_-$]}
as follows:
\smallskip

{\bf [A$_\pm$]}~{\em  For some constant $b_\pm >0$ and differentiable function
$f_\pm (x)$, one has }
\begin{equation}
 \label{fn_h}
 h(x)= \pm \exp\{ b_\pm  x^2 +  f_\pm (x) +o(1) \}  \quad \mbox{\em with \ $ f'_\pm (x)=o(x)  $ \ as \
 $x\to \pm \infty$}.
 %b(x)e^{cx^2},
\end{equation}

Now we can state the following key result on the tail behaviour.

\begin{thm}
 \label{prop_1} The following assertions hold true.

\smallskip

{\rm (i)}~Under condition {\bf [A$_+$]}, the stationary distribution $F$ of
$Y_t$ has a right tail of the form \eqref{rt} with
$\beta_+=1/(2b_+\sigma^2)=\alpha/(b_+\tau^2)$.

\smallskip

{\rm (ii)}~Similarly, under condition {\bf [A$_-$]} the left tail of $F$ is of
the form \eqref{lt} with $\beta_-=1/(2b_-\sigma^2)=\alpha/(b_-\tau^2)$.
\smallskip

{\rm (iii)}~Under condition {\bf [A$_+$]}, the right tail of the transition
distribution function $F_{s,t} (y_s, y_t):=\Pb (Y_t \le y_t \, |\, Y_s = y_s),
$ $0\le s < t,$ is regularly varying of index
\[
- \beta_{+,t-s} := -\frac{1}{2b_+\sigma_{t-s}^2}= - \frac{\alpha}{b_+\tau^2 ( 1 - e^{-2\alpha (t-s)})} <
-\beta_+.
\]
Under condition {\bf [A$_-$]}, a symmetric assertion holds for the left tail of
the transition distribution, with the regular variation index $- \beta_{-,t-s}
:= -(2b_-\sigma_{t-s}^2)^{-1}$.

{\rm (iv)}~If\/~{\bf [A$_+$]} is met and $h(-x) = o(h(x))$ as $x\to\infty$,
then, for any $0\le s < t,$ the right tail of the distribution of the increment
\/ $Y_t- Y_s$ in the stationary process has the same regularly varying
asymptotic behaviour at infinity as that of $Y_t$ from part~{\rm (i)}.
Similarly, if\/~{\bf [A$_-$]} holds  and $h(x) = o(h(-x))$ as $x\to\infty$,
then the right tail of the increment \/ $Y_t- Y_s$ in the stationary process
has the same regularly varying asymptotic behaviour at infinity as the\/ {\em
left} tail of $F$ from part~{\rm (ii)}.
\end{thm}

The proof of the theorem is given in Section~\ref{Sec_proofs}.

Now suppose that our $h$ is linear in the ``middle part": say, on an interval
$I$ such that $\Pb (X_t \in I)\ge 1-\ep/2$ (with a small enough $\ep$), we have
\[
h(x) = a + cx, \qquad x\in I,
\]
for some constants $a\in\R$, $c>0$, and then outside $I$ the function $h$ has
``tail parts" satisfying conditions~{\bf[A$_\pm$]}. Then, assuming that the log
price time series $\{Z_t\}$ is represented by (the values at integer time
points $t$ of) the
process $\{Y_t = h(X_t)\}$, %with a stationary $\{X_t\}$ driven by~\eqref{ou},
we will have for the log returns a representation of the form
\begin{equation}
 \label{h-h}
\Delta_\delta Z_t = h (X_t) - h (X_{t-\delta}) = c(X_t - X_{t-\delta})
 + \xi_{t,\delta},\qquad \Pb (\xi_{t,\delta}\neq 0)\le \ep,
\end{equation}
since $\{\xi_{t,\delta}\neq 0\}\subset\{X_t\not\in I\}\cup\{X_{t-\delta}\not\in
I\}.$ Therefore the ``central part" of the distribution of $\Delta_\delta Z_t$
will be very close to that of $ c(X_t - X_{t-\delta})$, which is clearly
normal. This shows that our model is consistent with empirical fact~{\bf
[F$_1$]}.

That {\bf [F$_2$]} is also reproduced by the model follows from
Theorem~\ref{prop_1}(iv). Here a curious observation is in order. On the one
hand, the log-log plots in Fig.~\ref{fig:lr} display ``shifted" parallel
straight line segments in the curves corresponding to different values of
$\delta$, and such a translation indicates that the tail behaviours of the
returns' distributions for different $\delta$'s  differ by constant factors
(increasing with $\delta$). On the other hand, Theorem~\ref{prop_1}(iv) claims
that they should have common asymptotics, which seems to be a contradiction.
However, large sample simulations of the differences of transformed components
of Gaussian vectors demonstrate the same translation of the straight line
segments present in the log-log plots of the empirical distribution tails for
increments, corresponding to different $\delta$ values, in the ``moderately
large" deviations zone (in agreement with the empirical observations), which is
then followed by a ``fusion" of the curves for larger deviation values thus
confirming the common asymptotics established in Theorem~\ref{prop_1}(iv). The
proof of the theorem suggests the following explanation for this phenomenon:
the established (common) asymptotics are essentially due to the convergence of
the conditional probability in the integral on the right-hand side
of~\eqref{I_32} (or, rather, the conditional probability $\Pb (Y_s < z-y \mid
Y_t =z)$ in~$I_3$, see~\eqref{III}) to one. However, in vicinity of the lower
integration limit $(1+\ep)y$ in $I_3$ this convergence is slow. In combination
with the regular variation of the right tail of the distribution of $Y_t$, this
results in a ``pre-limiting" tail beahaviour differing from that of $\Pb (Y_t
>y)$ by a factor depending on $\delta=t-s$ and $y$, which can be small for
small $\delta$ (hence the translation of the plots), but which eventually tends
to one as $y\to\infty$.

To illustrate the above statement concerning the simulation study,
Fig.~\ref{taila} displays the log-log plots  of the empirical distribution
tails for three i.i.d.\ samples (of size $3\times 10^5$ each) of the
differences $h(\zeta_2) - h(\zeta_1)$, where
\begin{equation}
 \label{h-h-H}
h(x) = 2(e^{  (3+{\rm sign}\, x)x^2/20} -1)\, \mbox{sign}\, x +  x, \qquad x\in\R,
\end{equation}
has the form stipulated in~\eqref{fn_h}, with $b_+=0.2$ and $b_-=0.1$, and
$(\zeta_1, \zeta_2)$ have standard normal bivariate distributions with
correlations $e^{-a}$, with $a=0.5$ (the left curve), 1 and 1.5 (the right
curve), respectively,  playing the role of $\alpha\delta$.

\begin{figure}[h]  % OK
\centering
   \includegraphics[width=8cm]{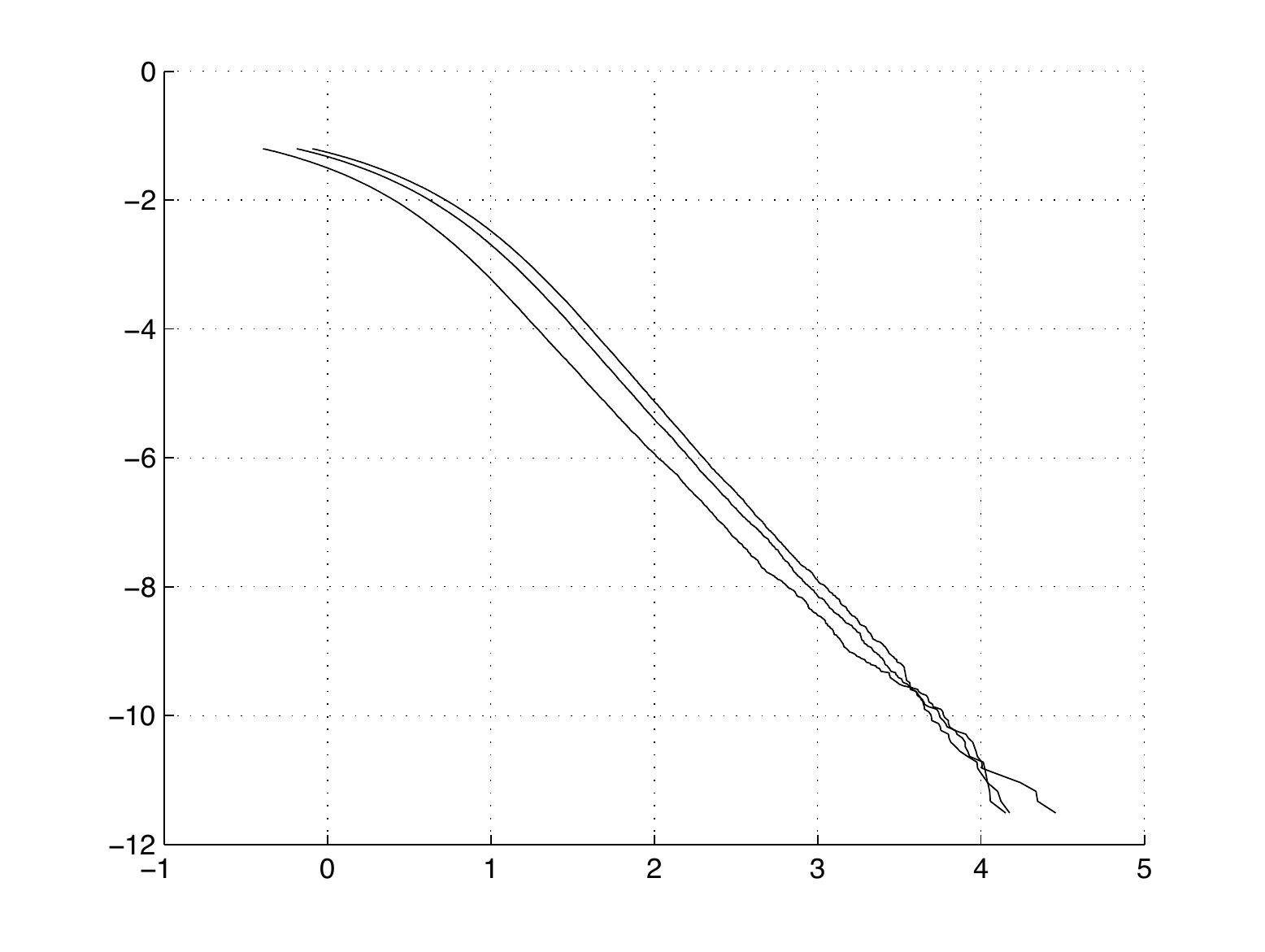}
 \caption{\small The log-log plots of the empirical distribution tails for simulated differences $h(\zeta_2) - h(\zeta_1)$ for i.i.d.\ standard Gaussian
 random vectors $(\zeta_1, \zeta_2)$ with different correlations
 (for detailed description, see the text around formula~\eqref{h-h-H}).}
\label{taila}
\end{figure}

That property {\bf [F$_3$]} holds for $\{Z_t :=h(X_t)\}$ follows from the
exponential form of the correlation function~\eqref{ou_corr}, the invariance of
Spearman's rho under strictly increasing transformations (so that $\rho_S
(h(X_s), h(X_t)) = \rho_S (X_s, X_t)$) and relation~\eqref{rhorho} (note that
the function of $\rho$ on the RHS of~\eqref{rhorho} is very close to the
identity function on the interval $\rho\in [-1,1]$).

The demonstration of {\bf [F$_4$]} is somewhat less straightforward. First we
state the following bound of which the proof is given in
Section~\ref{Sec_proofs}. As can be seen from the proof, the bound is rather
conservative, and so the result is more qualitative than quantitative in
nature: one can expect that, in most cases, the distance between the values of
the Spearman's rhos will be much smaller than the bound given.

\begin{thm}
 \label{prop_2}
Let $U$, $V$, $\xi$ and $\eta$ be random variables given on a common
probability space and such that $U$, $V$ and the sums
\[
U_1 := U+ \xi, \qquad  V_1 := V+ \eta
\]
are all continuously distributed. Then, setting   $\Pb (\xi\neq 0)=:\ep_1  $
and $\Pb (\eta\neq 0)=:\ep_2, $ one has
\[
|\rho_S (U_1, V_1) - \rho_S (U , V )| \le 18 (\ep_1 + \ep_2) + 12 \ep_1\ep_2.
\]
\end{thm}

From that bound and \eqref{h-h} it follows that
\[
\bigl|
 \rho_S (\Delta_\delta Z_t, \Delta_\delta Z_{t+k}) - \rho_S (\Delta_\delta X_t  , \Delta_\delta X_{t+k}   )
 \bigr|
 \le \ep_0,
\]
where $\ep_0 = 36\ep +12 \ep^2.$ Here $(\Delta_\delta X_t  , \Delta_\delta
X_{t+k}   )\equiv (X_t - X_{t-\delta}, X_{t+k} - X_{t+k-\delta} )$ has a joint
Gaussian distribution, and so in view of the above-stated
relationship~\eqref{rhorho} between $\rho_S$ and $\rho$ in the Gaussian case,
it remains to evaluate the linear correlation
\[
\rho   (\Delta_\delta X_t  , \Delta_\delta X_{t+k}   )
 = \frac{\E (X_t - X_{t-\delta})(X_{t+k} - X_{t+k-\delta} )}{\mbox{\rm Var} (X_t - X_{t-\delta})},
\]
where we used the stationarity and zero means of $\{X_t\}$. Expanding the
numerator and employing~\eqref{ou_corr}, we obtain, for $k\ge 0$ and $\delta >
0$,
\begin{align*}
\E (X_t - X_{t-\delta}) (X_{t+k} & - X_{t+k-\delta} ) \\
  &= \E  X_t X_{t+k } -\E  X_t X_{t+k-\delta} - \E X_{t-\delta} X_{t+k} + \E  X_{t-\delta} X_{t+k-\delta}\\
 &= \sigma^2 \bigl(2 \rho(X_t, X_{t+k}) - \rho(X_t, X_{t+k-\delta})- \rho(X_t, X_{t+k+  \delta})\bigr)\\
 &= \sigma^2 \bigl(2e^{-\alpha k} - e^{-\alpha |k-\delta|}- e^{-\alpha (k+\delta)}\bigr).
\end{align*}
As a special case of this where $k=0,$ we have $\mbox{\rm Var} (X_t -
X_{t-\delta})= 2\sigma^2 (1 - e^{-\alpha\delta})$, so that
\[
\rho   (\Delta_\delta X_t  , \Delta_\delta X_{t+k}   )
 =\frac{2e^{-\alpha k} - e^{-\alpha |k-\delta|}- e^{-\alpha (k+\delta)}}{2(1 - e^{-\alpha\delta})}.
\]
\begin{figure}[h]  % OK
\centering
   \includegraphics[width=8cm]{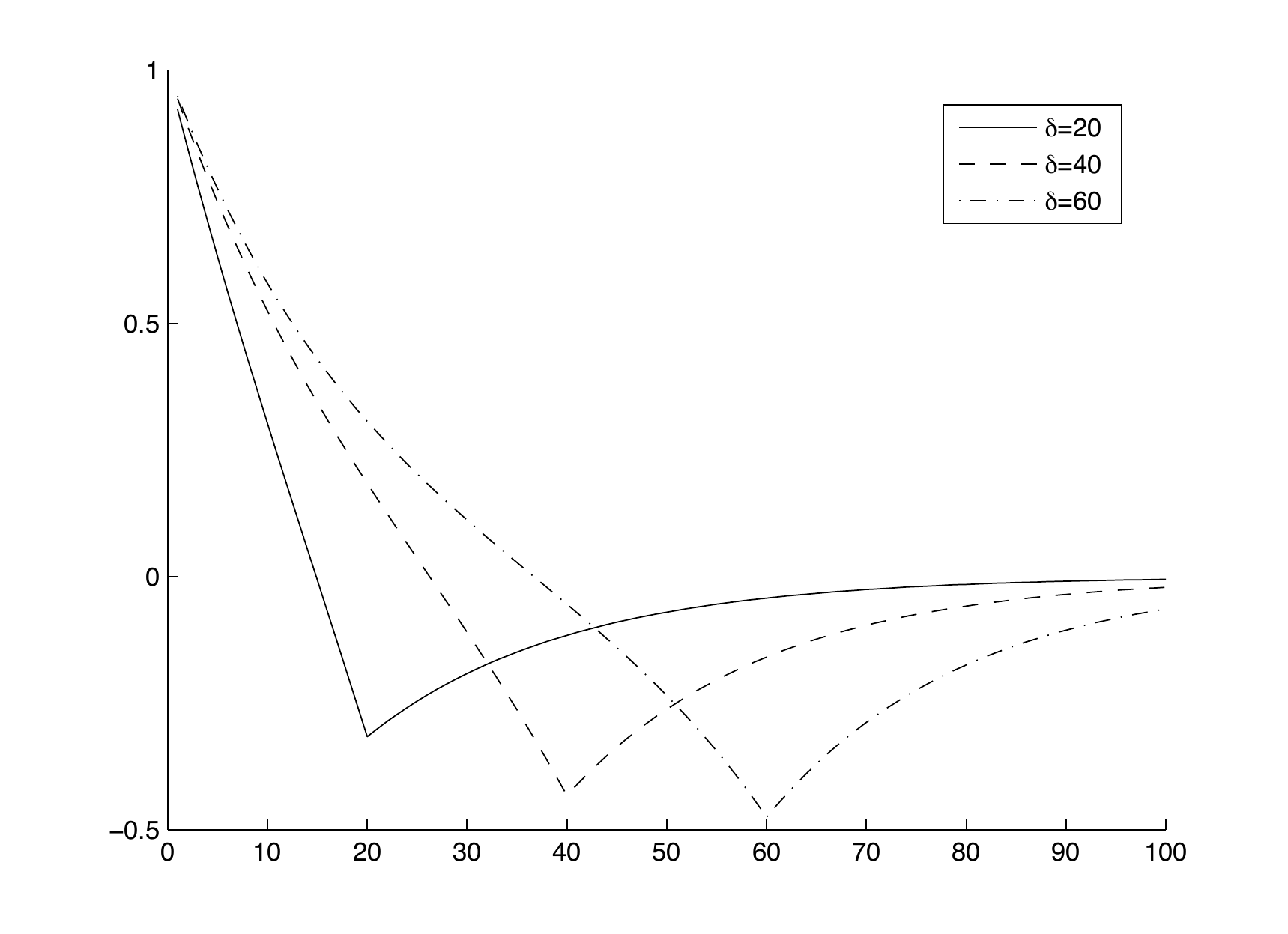}
 \caption{\small The plots of the correlation functions  $\rho (\Delta_\delta X_t, \Delta_\delta X_{t+k})$ in the case $\alpha=0.05,$
 for $\delta=20, 40$ and~60 days, plotted for $k\le 100$ (cf.\ Fig.~\ref{fig_incr_corr}). }
  \label{OU_diff_corr}
\end{figure}
The plots of this function of $k$ for three different values of $\delta$ are
presented in Fig.~\ref{OU_diff_corr}. They demonstrate virtually the same
behaviour as the empirical Spearman's rhos shown in Fig.~\ref{fig_incr_corr},
thus confirming that our model has property~{\bf [F$_4$]}.

Concerning~{\bf [F$_5$]}, we immediately see that, for our model $Y_t=h(X_t)$,
the vector $(Y_t, Y_{t+k})$ clearly has a Gauss copula. As it was the case when
dealing with~{\bf [F$_4$]}, showing that the copula for the vector of
increments $(\Delta_\delta Y_t, \Delta_\delta Y_{t+k})$ will be close to a
Gauss one, provided that the value of $\ep$ in~\eqref{h-h} is small, is a more
sophisticated task. However, that can be easily done adapting the proof of
Theorem~\ref{prop_2} and its extension to the case of bivariate distribution
functions. This is a straightforward technical exercise that we leave to the
interested reader.

\begin{rema}{\rm A possible  relatively simple modelling approach alternative to ours would be to use,
instead of~\eqref{ou1}, a stochastic driver of the form
\begin{equation}
 \label{ou_L}
dX_t = -\alpha X_t dt + dL_t,
\end{equation}
where $\{L_t\}$ is a suitably chosen L\'evy process. In particular, Theorem~3.2
in~\cite{HeLe} establishes the following result: {\em Let $\mu\in\R$ and
$\nu,\lambda>0$ be arbitrary,
\begin{align*}
 K_\nu (x) & = \frac12 \int_0^\infty y^{\nu -1}
 \exp \left\{-\frac{x}2 \left( y + y^{-1}\right)\right\}dy
  \\
  & \equiv \frac12 \int_{-\infty}^\infty
 \exp \left\{\nu u - x\cosh u\right\}du, \qquad x>0,
\end{align*}
be the modified Bessel function of the third kind, and $\{L_t\}_{t\ge 0}$ be a
L\'evy process with cumulant function
\[
\ln \E e^{iu L_1} = iu\mu
 - \lambda |u| \frac{K_{\nu/2-1} (\lambda |u|)}{K_{\nu/2}(\lambda |u|)},
  \qquad u\in \R\setminus \{0\}.
\]
Then the stochastic differential OU type equation\/~\eqref{ou_L}   has a weak
strictly stationary solution satisfying
\[
X_t = e^{-\alpha t} X_0 + e^{-\alpha t} \int_0^t e^{ \alpha s} dL_s
\]
and such that $X_t$ has marginal $t$-distribution $T(\nu, \lambda, \mu)$ with
density
\[
\frac{c(\nu, \lambda)}{[1+ ((x-\mu)/\lambda)^2]^{(\nu + 1)/2}}, \qquad x\in \R,
  %\quad c(\nu, \delta) = \frac{\Gamma((\nu+1)/2)}{\delta \sqrt{\pi}\Gamma(\nu/2)}
  %\equiv \frac1{\delta {\rm B}(\nu/2, 1/2)}
\]
where $c(\nu, \lambda) = (\lambda {\rm B}(\nu/2, 1/2))^{-1}$, ${\rm
B}(\cdot,\cdot)$ being the beta function. Moreover, if $\nu
>1$ then $\E X_t =\mu$, and if $\nu >2$ then $\mbox{\rm corr}\, (X_{t+\tau}, X_t)=
e^{-\alpha |\tau|}$.}

As the resulting process $\{X_t\}$ will already have power function decay of
the distribution tails, it might be possible to use that process for modelling
the financial data of interest, without transforming it (although some
transformation might still be needed, due to the symmetry of $t$
distributions). The advantage of using a transformed OU model of the type dealt
with in the present paper is that it is more flexible as its temporal
dependence structure can be {\em decoupled\/} from its distribution tail
properties.

}
\end{rema}

\section{Further comments on model fitting}
\label{sec_com}

First we observe that, without loss of generality, one can always put $\tau=1$
in~\eqref{ou1}: indeed, we clearly have $Y_t=\tilde h(\tilde X_t)$ with $\tilde
h(x):=h(\tau x)$ and $\tilde X_t:= \tau^{-1} X_t$ satisfying
\[
d \tilde X_t = -\alpha \tilde X_t dt +   dW_t,  \qquad t\ge 0.
\]
So in the ``driver" component~\eqref{ou1} of the model, one only needs to
estimate $\alpha$, and it is this task that will be discussed in the present
section.

Concerning the estimation of the function $h$, we will only point out that, in view of
its specific form dictated by Theorem~\ref{prop_1}, a parametric approach would be most
natural. One possible simple parametric class of functions consists of candidates of
the following spline type:
\[
h(x)= - a_1 e^{b_- x^2}  {\bf 1} (x\le x_1)
  +  ( a_2 + a_3 x ) {\bf 1} (x_1 < x\le x_2)
  +  a_4 e^{b_+ x^2} {\bf 1} (x>x_2),
\]
where ${\bf 1}(\cdot) $ denotes the indicator function, the bounds $x_i$ can be
estimated as the end points of the ``central part" of the distribution well
approximated by a normal one, $a_2$ and $a_3>0$ can be estimated from that ``central
part" (i.e.\ from a truncated sample), parameters $b_\pm >0$ can be estimated using
estimates for the tail regular variation indices (e.g.\ Hill estimators) and the
assertions of Theorem~\ref{prop_1} (note that, to use them, one has to
estimate~$\alpha$ first), and $a_1 $ and $a_4$ can be chosen to make the function $h$
continuous. Alternative parametric families may be more adequate.

Suppose that the process $\{Y_t, t\in[0,T]\}$ is observed at times
$t_k=k\epsilon$, $k=0,\hdots,n$, where $n+1$ is the total number of
observations. Without loss of generality, we can always assume that
$\epsilon=1$.

\subsection{An estimator based on rank correlations}

%The vector of observations $\{Y_0,Y_1,\hdots,Y_n\}$ has a Gauss copula.

Since Spearman's rho is invariant under strictly increasing transformations of
data and $(X_j, X_k)$ is Gaussian, we have from~\eqref{rhorho} that
\[
\rho_S (Y_j, Y_k) = \rho_S (X_j, X_k) = \frac6{\pi} \arcsin \frac{\rho (X_j, X_k)}2
 =  \frac6{\pi} \arcsin \frac{e^{-\alpha |j-k|}}2,
\]
so that
\[
 \alpha  = -\frac{1}{|j-k|} \ln \left(2 \sin\left(\frac{\pi}6  \rho_S (Y_j, Y_k)\right)\right).
\]
Hence one can use the ``plug-in" estimators
\[
\hat\alpha_k:= -\frac{1}{k} \ln \left(2 \sin\left(\frac{\pi}6 \hat\rho_S (Y_0, Y_k)\right)\right), \qquad k=1,2,\ldots,
\]
where $\hat\rho_S (Y_0, Y_k)$ is calculated according to~\eqref{rhoSest}, for
the sample of pairs~$(Y_t, Y_{t+k})$. As discussed before, such an estimation
would require blocking of data, to reduce the effect of the slow trend
component in the time series. Alternatively, one can use Spearman's rho for the
increments of $Y_t$ (cf.\ our discussion following Theorem~\ref{prop_2}). The
values of $\hat\alpha_k$ for different $k$'s obtained from 50 day log-returns
for BHP data are shown in Table~\ref{tab_2}. Note the stability of the
estimates, which confirms the appropriateness of the model (according to which
the quantities are estimators of the same value).

\begin{table}[h]
\centering
\begin{tabular}{|c|c|c|c|c|c|c|}
\hline
    $k$ &  5 & 10 & 15 & 20 & 25 & 30 \\
     \hline
$\hat\alpha_k$  & 0.032  & 0.036 & 0.036 & 0.037 & 0.036 & 0.038\\
 \hline
\end{tabular}
 \caption{\small The values of $\hat\alpha_k$ from 50 day
 BHP stock log returns, for
 different values of~$k$.}
\label{tab_2}
\end{table}

\subsection{An alternative estimator}
\label{sect:est}

An alternative (but closely related) estimator for $\alpha$ is based on the following
simple observation concerning the background OU process, which we will formulate in the
form of a theorem (for its proof, see Section~\ref{Sec_proofs}). Set
$$
g(x) := -\ln\left[\sin\left(2\pi(x-1/4)   \right) \right], \qquad  x\in (1/4, 1/2).
$$

\begin{thm}\label{prop_3}
Let
\[
\hat{p}_k^0 := \frac{1}{n-k}\sum\limits_{j=0}^{n-k}\mathbf{1}(X_j>0, X_{j+k}>0)
\]
be an empirical estimate of $p_k:=\Pb(X_0>0, X_k>0)$. Then
\begin{equation}
\label{est}
\hat \alpha^0_k : =\frac{ g(\hat{p}^0_k)}{k}, \qquad k=1,2,\ldots,
\end{equation}
are strongly consistent and asymptotically normal\/ $($as $n\to\infty)$
estimators of $\alpha$ with
\begin{equation*}
 %\label{variancealpha}
{\rm Var}(\hat \alpha^0_k) \sim \frac{v^2_k (g'(p_k))^2}{n k^2},
\end{equation*}
where
\begin{equation}\label{variance}
v^2_k := p_k(1-p_k) + 2\sum\limits_{l=1}^\infty \left( \Pb(X_{k+l}>0, X_l>0, X_k>0,
X_0>0) - p_k^2 \right).
\end{equation}
\end{thm}

Of course, in the context of our problem $\{X_t\}$ is unobservable, so we need
to replace $\hat p^0_k$ with an estimator for $p_k$ from $\{Y_t\}$. Denoting by
$m$ the median of $F$, we clearly have $\{X_j >0\}=\{Y_j>m\}$ and so
\[
\hat{p}_k^0 = \frac{1}{n-k}\sum\limits_{j=0}^{n-k}\mathbf{1}(Y_j>m   , Y_{j+k}>m ).
\]
We still do not know $m$, but one can use instead the empirical median $\hat
m_n (Y)$ of the sample $Y_0, Y_1, \ldots, Y_n$. It is not hard to verify that
$\hat m_n (Y)\stackrel{{\rm a.s.}}{\longrightarrow} m$ as $n\to\infty$, which
implies that
\[
\hat{p}_k := \frac{1}{n-k}\sum\limits_{j=0}^{n-k}\mathbf{1}(Y_j>\hat m_n(Y), Y_{j+k}>\hat m_n(Y) )
 \stackrel{{\rm a.s.}}{\longrightarrow} p_k.
\]
Therefore
\[
\hat \alpha'_k : =\frac{ g(\hat{p}_k)}{k}, \qquad k=1,2,\ldots,
\]
will also be strongly consistent estimators of~$\alpha$.

The values of $\hat\alpha'_k$ for different $k$'s obtained from 50 day
log-returns for BHP data are shown  in Table~\ref{tab_3}.

\begin{table}[h]
\centering
\begin{tabular}{|c|c|c|c|c|c|c|}
\hline
    $k$ &  5 & 10 & 15 & 20 & 25 & 30 \\
     \hline
$\hat\alpha'_k$  & 0.027  & 0.035 & 0.035 & 0.032 & 0.031 & 0.031\\
 \hline
\end{tabular}
 \caption{\small The values of $\hat\alpha'_k$ from 50 day BHP stock log returns, for
 different values of~$k$.}
\label{tab_3}
\end{table}

Simulations showed that the  root mean square errors  of the estimators $\hat p_k$ are
smaller than those for $\hat{p}_k^0$, so that from the viewpoint of the quadratic error
minimisation it is better, in a sense, not to know the true value of the median of the
stationary distribution~$F.$ Figure~\ref{fig:ratio} displays the values of the ratio of
the standard error of $\hat{p}_k^0$ to that of $\hat{p}_k$, for $k=20$,   calculated
from a sample of $10^4$ independent discretised trajectories of the OU process on the
time interval $[0,T]$, with $T=10^4$ and discretisation step $\epsilon = 1.$

\begin{figure}[ht]
\centering
\includegraphics[width=8cm]{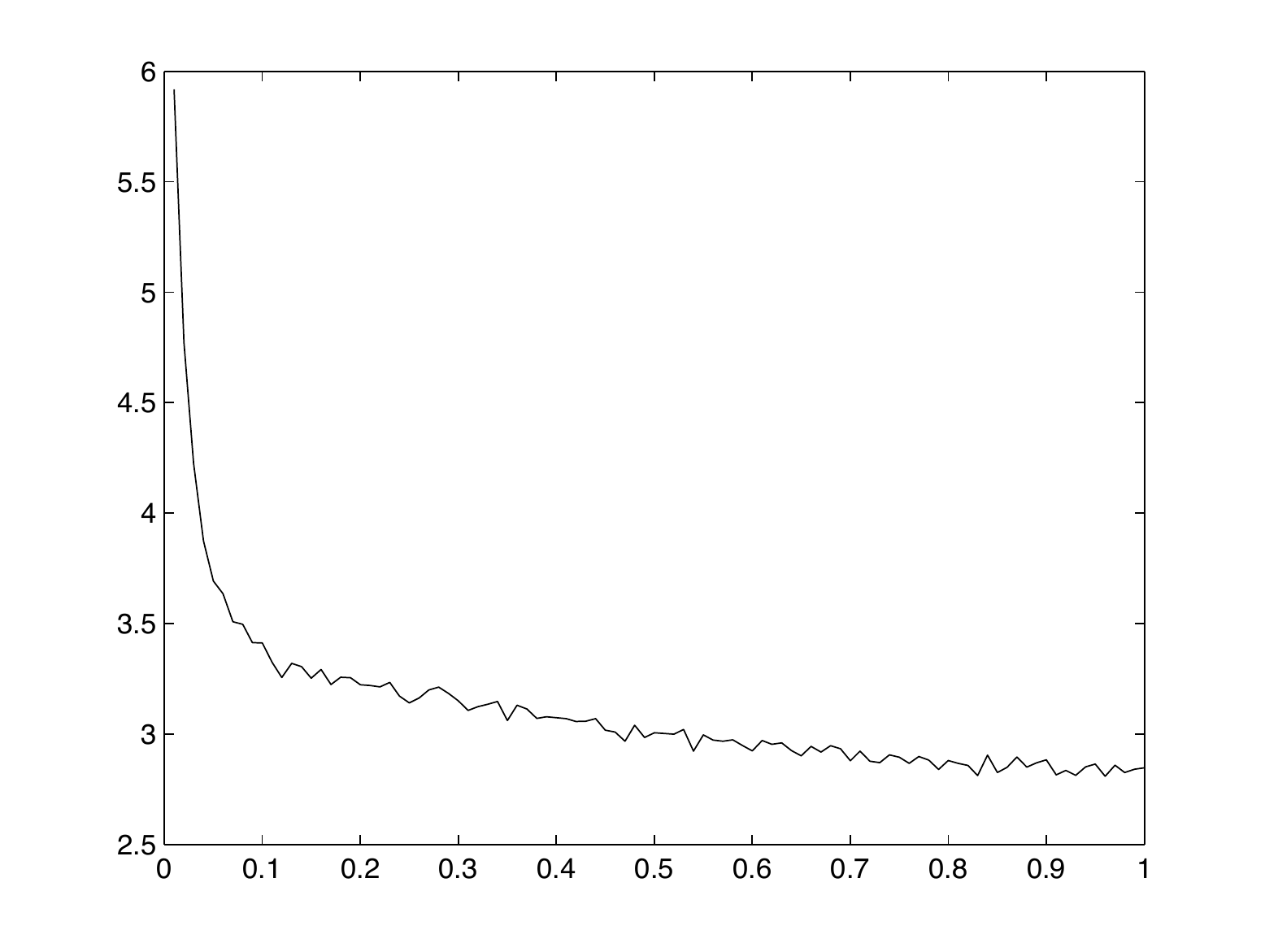}
\caption{The ratio of the root mean square errors of $\hat{p}_{20}^0$ and $\hat{p}_{20}
$ as a function of $\alpha$ calculated using crude Monte-Carlo simulations of the
respective OU process. }
 \label{fig:ratio}
\end{figure}

It can be seen from the plot that, as $\alpha$ increases, the value of the
ratio appears to decrease to some limiting value. That limiting value should
correspond to the case of independent observations $Y_j$, $j=0,1,2,\ldots, n$
(or $X_j$), in which the ratio can be calculated, as shows Theorem~\ref{prop_4}
below (its proof is given in Section~\ref{Sec_proofs}). Unfortunately, we were
not able to quantify the effect in the general case, for finite values of
$\alpha$, but a plausible ``common sense explanation" of it remains the same:
the ``quality" of the estimators $\hat{p}_k^0$ and $\hat{p}_k$ is basically
determined by how often the difference $X_j - a$ changes its sign, for $a=0$
and $a=\hat m_n(X)$, respectively, and in the latter case this occurs more
often.

For notational convenience, in Theorem~\ref{prop_4} we will consider samples
starting from $X_1$ (rather than $X_0$), so that here we will have
\begin{align*}
\hat{p}_k^0 &:= \frac{1}{n-k}\sum\limits_{j=1}^{n-k}\mathbf{1}(X_j>0, X_{j+k}>0), \\
\hat{p}_k   &:= \frac{1}{n-k}\sum\limits_{j=1}^{n-k}\mathbf{1}(X_j>\hat m_n(X),
X_{j+k}>\hat m_n(X) ),
\end{align*}
$\hat m_n(X)$ being the sample median for $X_1, X_2, \ldots, X_n$.

\begin{thm}
 \label{prop_4}
Let $X_1,X_2,\ldots, X_n$ be i.i.d.\ standard normal random variables and
$p_k:=\Pb(X_j>0, X_{j+k}>0)$, $k=1,2,\ldots$  Then, for any fixed~$k$,
$\hat{p}_k^0$ is an unbiased estimator of~$p_k=1/4$, $\hat{p}_k$ has a bias of
$O(n^{-1})$ as $n\to\infty$, and
\[
 {\rm Var\,}(\hat{p}_k^0)=\frac{5}{16n} + O(n^{-2}), \qquad
 {\rm Var\,}(\hat{p}_k) =\E (\hat{p}_k -p_k)^2 + O(n^{-2})=\frac{1}{16 n} + O(n^{-2})  .
\]
\end{thm}

\subsection{An estimator based on band crossing times}

In the two previous subsections we considered estimation of the drift
coefficient $\alpha$ of the unobservable OU process $\{X_t\}$ from its
transformed version $\{Y_t=h(X_t)\}$ which was assumed to be observable. It is
interesting to note that one can estimate $\alpha$ even if $\{Y_t\}$ is only
partially observable itself\,---\,namely, if we only know the number of times
the process' trajectory crossed the band between two given levels $A$ and $B$
during a given time interval, provided that we also know the values of the
d.f.~$F(y)$ of the stationary distribution of $Y_t$ at the points $y =A,B$.

For $a,b\in\R$ denote by $T_{a,b}$ the time it takes the OU process $\{X_t\}$ driven
by~\eqref{ou1} with $\tau=1$ to transit from point~$b$ to point~$a$: slightly abusing
notation, one can write
\[
T_{a,b} := \inf\{t>0: X_t = a \mid X_0=b\}.
\]
The density of $T_{a,b}$ has a simple closed-form representation when
$a=0$~\cite{Sa}; in the general case, this is not so (one needs to invert a
Laplace transform expressed in terms of parabolic cylinder functions, see
e.g.~\cite{RiSa}). However, for our purposes it will be enough to know the
means of~$T_{a,b}$ only, and these are available, see~\cite{RiSa}.

By the excursion time $T_{(a),b}$ to $a$ from $b$ we will call the time it takes the OU
process starting from~$b$ to reach the point $a$ and then return back to~$b$. Clearly,
\[
T_{(a),b} \stackrel{d}{=} T_{a,b} + T_{b,a},
\]
where the random variables on the right-hand side are assumed to be independent.

Set $a:=h^{-1} (A) $, $b:= h^{-1} (B)$. Since $X_t\sim N(0, 1/(2\alpha))$ in
the stationary regime, one has $\Phi (\sqrt{2\alpha} a) = \Pb (X_t\le a)= \Pb
(h(X_t)\le A)=F(A)$ and likewise $ \Phi (\sqrt{2\alpha} b) =  F(B)$. Therefore
\[
 a= \frac{\Phi^{-1} (F(A))}{\sqrt{2\alpha}}=:\frac{a^*}{\sqrt{\alpha}}, \qquad
 b= \frac{\Phi^{-1} (F(B))}{\sqrt{2\alpha}}=:\frac{b^*}{\sqrt{\alpha}}.
\]

The standard scaling argument shows that the process $X^*_t:= \alpha^{1/2}
X_{t/\alpha}$ satisfies \eqref{ou1} with $\alpha=1,$ $\tau =1$ (and another
Brownian motion process, namely $W^*_t=\alpha^{1/2} W_{t/\alpha}$, instead of
$W_t$). Therefore one has
\[
T_{(a),b} \stackrel{d}{=} \alpha^{-1} T^*_{(a^*),b^*},
\]
where the asterisk indicates that the excursion time on the right hand side is
for the process $\{X^*_t\}$, the distribution of $T^*_{(a^*),b^*}$ being
independent of $\alpha$ and having a mean $m_{a^*,b^*}$ which can be computed
from $F(A)$ and $F(B)$ using the results of~\cite{RiSa}.

Now the strong Markov property implies that the subsequent excursions from $b$
to $a$ in $\{X_t\}$ (or, equivalently, from $B$ to $A$ in $\{Y_t\}$) form an
i.i.d.\ sequence. Therefore, if $N_{A,B} (T)$ denotes the number of times
$\{Y_t\}$ crosses the band between the levels $A$ and $B$ in a given direction
(say, from~$B$ to~$A$, each such crossing corresponding to an excursion
from~$B$ to~$A$) during the time interval $[0,T]$, the integral renewal theorem
asserts that, as $T\to\infty$,
\[
\frac{N_{A,B} (T)}{T}
 \stackrel{{\rm a.s.}}{\longrightarrow} \frac1{\E T_{(a),b}}
 =  \frac{\alpha}{m_{a^*,b^*}},
\]
which leads to another consistent estimator of $\alpha$ given by $ \hat
\alpha'':= m_{a^*,b^*} N_{A,B} (T) T^{-1}.$ In conclusion we note that one can
evaluate the mean quadratic error of the estimator $ \hat \alpha''$ from the
central limit theorem for renewal processes (the variance of  $T_{a,b}$ is also
available from~\cite{RiSa}).

\section{Proofs}
\label{Sec_proofs}

\begin{proof}[Proof of Theorem~\ref{prop_1}]
(i), (ii)~As part~(ii) is a ``mirror reflection" of part~(i), we only need to prove the
latter (in fact, part~(ii) was included in the theorem as a separate assertion only for
convenience of referencing to it in part~(iv)).

Set $x:=  h^{-1}(y).$ Since $X_t \sim N(0, \sigma^2)$, we can use Mills' ratio for the
normal tail and the relation $b_+\beta_+ = 1/(2\sigma^2)$ to write that, as
$y\to\infty,$
\begin{align}
\Pb(Y_t > y) &= \Pb( X_t >x) = \frac{\sigma}{x \sqrt{2\pi}}e^{-x^2/(2 \sigma^2)} (1+
o(1))
   \notag\\
 &= \frac{\sigma}{x\sqrt{2\pi}}\left(e^{b_+x^2}\right)^{-\beta_+}\left(1+ o(1)\right)
  = \frac{\sigma}{x\sqrt{2\pi}}\, e^{\beta_+ f_+ (x) + o(1)}(h(x))^{-\beta_+}
   \notag\\
 &= \frac{\sigma}{ \sqrt{2\pi}}\, e^{\beta_+ f_+ ( x) - \ln x  + o(1)}y^{-\beta_+}
\label{tail}.
\end{align}
It remains to show that $g(y):=e^{f_0 (x)}$, where $f_0 (x) =\beta_+ f_+ ( x) - \ln x$
and $x= h^{-1}(y),$ is a slowly varying function of $y$. To this end, fixing an
arbitrary  $\lambda >1$ and setting $s:= h^{-1} (\lambda y)-x $, we can write
\[
\frac{g(\lambda y)}{g(y)} = \exp\{f_0 (x + s) - f_0(x)\}.
\]
Since clearly $f_0'(x)= o(x)$ as $x\to \infty,$ it now suffices to show that $s=O(1/x)$
.

From the definition of $s$, the mean value theorem and the assumption $f'_+ (x) =o(x),$
we have
\begin{align*}
\lambda & = \frac{h(x+s)}{h(x)}= \exp\{ 2b_+ xs + b_+ s^2 + f_+ (x+s) - f_+ (x) + o(1) \}\\
 & = \exp\{ 2b_+ xs (1 + o(1)) + bs^2 +  o(1) \},
\end{align*}
which immediately implies that $s= O(1/x)$ and hence completes the proof of part~(i).

\smallskip

(iii)~We will only consider the case where {\bf [A$_+$]} holds true. The case of {\bf
[A$_-$]} is dealt with in the same way.

First observe that, given $Y_s=y_s$ (or, equivalently, $X_t = x_s :=h^{-1} (y_s)$), we
can write the increment of the process $\{Y_t\}$ on $[s,t]$ as
\[
Y_t - y_s = h(X_t ) - h(x_s)
 = h_1( X_t - m_{t-s} (x_s)) - h(x_s)  ,
\]
where, setting for brevity $m:=m_{t-s} (x_s)$, the function
\[
h_1 (x) = h (x + m)= \exp\{bx^2 + 2 bmx + m^2 + f(x+m) + o(1)\}
 =: e^{bx^2   + f_1(x ) + o(1)}
\]
is again of the form \eqref{fn_h}, with  $f_1 (x): =2 bmx + m^2 + f(x+m)$ clearly
satisfying the relation $f_1'(x) = o(x).$ Since, as we have already recalled, the
conditional distribution of $X_t - m_{t-s} (x_s)$ given $X_t = x_s$ is $N(0,
\sigma_{t-s}^2)$, the proof of~(iii) is completed by the same argument as used to
demonstrate the first part of the theorem.

\smallskip

(iv)~We will prove the first part of the assertion, assuming that {\bf [A$_+$]} is met
and $h(-x) = o(h(x))$ as $x\to\infty$. The second half will then immediately follow, in
view of the observations that $Y_t-Y_s = - (Y_s - Y_t)$ and $(Y_s, Y_t) \stackrel{d}{=}
(Y_t, Y_s)$, the latter being a consequence of the obvious relation $(X_s, X_t)
\stackrel{d}{=} (X_t, X_s)$ (for the
stationary OU process, of course). %%

Letting $y\to\infty$ and $\ep=\ep (y)\downarrow 0$ slowly enough (so that, in
particular, $\ep y\to \infty$; further conditions on the behaviour of $\ep$ will appear
later), we can write
\begin{align}
\Pb (Y_t - Y_s >y)& = \int  \Pb (Y_s< z - y , \, Y_t \in dz)
 \notag\\
 &=\int_{-\infty}^{(1-\ep) y} +  \int_{(1-\ep) y}^{(1+\ep) y} +\int_{(1+\ep) y}^{ \infty}
 =: I_1 +  I_2 +I_3 .
 \label{III}
\end{align}
First consider
\begin{align}
I_1 &= \int_{-\infty}^{(1-\ep) y}  \Pb (Y_s< z - y ,\,  Y_t \in dz)
  \le \int_{-\infty}^{(1-\ep) y}  \Pb (Y_s< -\ep y ,\,  Y_t \in dz)
 \notag \\
 & = \Pb (Y_s< -\ep y ,\,  Y_t  \le (1-\ep) y)
   \le  \Pb (Y_s< -\ep y) = \Pb ( - h(X_s) > \ep y).
  \label{I_1}
\end{align}
In the case where $\inf_x h(x) >-\infty$ one clearly has $I_1=0$ for all large
enough~$y$. In the alternative case, the function  $v(w):= - h^{-1} (-w)$ will be
defined for all $w \in\R$ and have the property $v(w)\to\infty$ as $w\to\infty$.
Moreover, due to the assumption that $h(-x) = o(h(x))$ as $x\to\infty,$ there exists a
number $x_0\in\mathbb{R}$ and a continuous function $\eta (x)\ge 0$ vanishing at
infinity such that $h(-x) =-\eta (x) h(x)$ for all $x\ge x_0$; note that $\eta^* (x)
:=\sup_{u\ge x} \eta(u)\downarrow 0$ as $x\to\infty$. From that and the symmetry of the
distribution of $X_s$ we see that, for all large enough $y$, the last probability
in~\eqref{I_1} is equal to
\begin{align*}
\Pb ( - h(-X_s) > \ep y, X_s >v(\ep y))
 &= \Pb ( \eta (X_s) h(X_s) > \ep y,\, X_s >v(\ep y))
 \\
 & \le \Pb \bigl( \eta^* (v(\ep y)) h(X_s) > \ep y\bigr),
\end{align*}
and hence
\[
I_1 \le \Pb \left(  h(X_s) > \frac{\ep}{\eta^* (v(\ep y))}\, y\right).
\]
Now since we can assume that $\ep\to 0$ so slowly that $\ep/\eta^* (v(\ep y))\to\infty$
(as such a choice is clearly possible),  it follows from part~(i) that $I_1 =
o(\overline{F} (y))$.

Further, it is obvious from the Uniform Convergence Theorem for slowly varying
functions (see e.g.\ Theorem~1.1.2 in~\cite{BoBo}) that
\begin{equation}
 \label{equiv}
((1\pm \ep) y)^{-\beta_+  } L_+ ((1\pm\ep) y) = ( 1+ o(1)) y^{-\beta_+  } L_+ (y),
\end{equation}
and so from (i) one immediately obtains the bound
\begin{align}
I_2  &\le \Pb \bigl(Y_t \in ((1-\ep) y, (1+\ep) y]\bigr) = \overline{F} ( (1-\ep) y)
- \overline{F} ((1+\ep) y) \notag\\
 & = (1+o(1)) y^{-\beta_+  } L_+ (y) - (1+o(1)) y^{-\beta_+  } L_+ (y)
  = o(\overline{F} (y)).
 \label{I_2}
\end{align}

It remains to evaluate $I_3$. To this end, we note that, on the one hand, in view of
\eqref{equiv} and the result of part~(i), one has
\begin{align}
 \label{I_31}
I_3 %& = \int_{(1+\ep) y}^{\infty}  \Pb (Y_s< z - y \,|\, Y_t =z) \,\Pb (Y_t \in
   %dz) \\ &
  \le \int_{(1+\ep) y}^{\infty} \Pb (Y_t \in dz)
 = \overline{F} ((1+\ep) y)  = (1+ o(1))  y^{-\beta_+  } L_+ (y),
\end{align}
while on the other hand,  as $1-y/z \ge \ep /(1 + \ep)$ for $z \ge (1 + \ep )y$,
\begin{align}
  \label{I_32}
I_3  \ge
 \int_{(1+\ep) y}^{\infty}  \Pb (Y_s/z< \ep /(1 + \ep) \,|\, Y_t =z) \,\Pb (Y_t \in dz).
\end{align}
Setting $u:=h^{-1} (z)$ and recalling that the conditional distribution of $X_s$ given
$Y_t= z$ (or, equivalently, $X_t =u$) coincides with the law of $W:=ue^{-\alpha (t-s)}
+ \sigma_{t-s}Z$, $Z\sim N(0,1),$ we see that the conditional probability in the last
integral is equal to
\begin{multline*}
 %\label{probA}
\Pb \biggl( b_+ (W^2 - u^2)  + f_+ (W)
 - f_+ (u) + \theta (W) - \theta (u) < \ln\frac{\ep}{1+\ep}  \biggr)
 \\
  = \Pb \biggl( b_+ (2ue^{-\alpha (t-s)}\sigma_{t-s} Z +\sigma_{t-s}^2 Z^2)  + r(W,u)
   < b_+ u^2 (1 - e^{-2\alpha (t-s)}) +  \ln\frac{\ep}{1+\ep}  \biggr),
\end{multline*}
where $\theta (u)= o(1)$ as $u\to\infty$, $r(W,u):= f_+ (W)
 - f_+ (u) + \theta (W) - \theta (u)$. Recalling that  $f'_+ (x)=o(x)$ at
infinity, it is obvious that the last probability  tends to one uniformly in $ z\ge
(1+\ep) y $  as $y\to\infty$, provided that we choose $\ep \to 0$ slowly enough
(ensuring that $\sqrt{|\ln \ep |}\ll h^{-1} (y) $). This, together with~\eqref{I_31}
and~\eqref{I_32}, shows that $I_3 = (1+ o(1)) y^{-\beta_+  } L_+ (y)$ and so, in view
of the earlier bounds for $I_1$ and $I_2$, completes the proof of the theorem.
\end{proof}

\begin{proof}[Proof of Theorem~\ref{prop_2}]
Since, for a continuous random variable $X$,  $F_X (X)$ is uniformly distributed on
$(0,1)$, we see from~\eqref{rhoS}   that
\begin{equation}
 \label{rhoS1}
\rho_S (U_1, V_1) = 12 \bigl(\E F_{U_1} (U_1)F_{V_1} (V_1) -0.25\bigr).
\end{equation}
Now, using the standard argument we see  that, for any $x\in\R,$
\begin{align*}
F_{U_1} (x) &\ge  \Pb (U+ \xi\le x;\, \xi =0) \ge \Pb (U  \le x ) - \Pb ( \xi
\neq 0) \ge F_U(x) -\ep_1,\\
F_{U_1} (x) & \le \Pb (U+ \xi\le x;\, \xi =0) + \Pb ( \xi \neq 0)
 \le \Pb (U  \le x ) + \ep_1 =F_U(x) +\ep_1,
\end{align*}
and similarly for $F_{V_1}(x),$ so that
\[
\sup_x |F_{U_1} (x)-F_{U } (x)|\le \ep_1, \qquad \sup_x |F_{V_1} (x)-F_{V } (x)|\le
\ep_2.
\]
Therefore, using ${\mathbf 1}(A)$ for the indicator function of the event $A$, we have
\begin{align*}
F_{U_1} (U_1)F_{V_1} (V_1)
  & \le F_{U_1} (U_1)F_{V_1} (V_1) {\mathbf 1} (\xi=\eta=0) + {\mathbf 1} (\xi \neq 0)+ {\mathbf 1} (\eta \neq 0)\\
  & = F_{U_1} (U )F_{V_1} (V ) {\mathbf 1} (\xi=\eta=0) + {\mathbf 1} (\xi \neq 0)+ {\mathbf 1} (\eta \neq 0)\\
  &\le (F_{U } (U  ) +\ep_1)( F_{V } (V  )+\ep_2)   + {\mathbf 1} (\xi \neq 0)+ {\mathbf 1} (\eta \neq 0),
\end{align*}
so that
\begin{align*}
\E F_{U_1} (U_1)F_{V_1} (V_1)
  &\le  \E (F_{U } (U  ) +\ep_1)( F_{V } (V )+\ep_2) + \ep_1 + \ep_2\\
  &\le \E  F_{U } (U  )  F_{V } (V )  + \frac32(\ep_1 + \ep_2)  + \ep_1  \ep_2,
\end{align*}
where we used the simple fact that $\E  F_{U } (U  ) =\E F_{V } (V ) =1/2$. This,
together with a similar lower bound (that can be obtained in exactly the same way) and
relation~\eqref{rhoS1}, completes the proof of the theorem.
\end{proof}

\begin{proof}[Proof of Theorem~\ref{prop_3}]
Estimator \eqref{est} is obtained as a method-of-substitution (``plug-in") estimator
from the relation $\alpha = g(p_k)/k$ which is established by inverting
\begin{equation}
 \label{proba}
p_k = \frac{1}{4} + \frac{1}{2\pi}\arcsin\left(e^{-\alpha k}\right)\in (1/4, 1/2),
\qquad k\ge 1,
\end{equation}
the latter representation following e.g.\ from Theorem~5.35 in~\cite{McNeil}. Since
$\hat{p}^0_k \stackrel{\rm a.s.}{\longrightarrow} p_k$ as $n\to\infty$ due to the
ergodicity of the OU process (see e.g.\ Theorem 5.6 in~\cite{KarlinT}), the consistency
of $\hat \alpha^0_k$   follows immediately from the continuity of $g$ (see e.g.\
Theorem~1 in Section~1.5 of~\cite{Borov}).

The OU process is $\alpha$-mixing  with exponential rate~\cite{Protter,Masuda}.
Therefore  $\{Z_j:=\textbf{1}(X_j>0,X_{j+k}>0)\}$ is also $\alpha$-mixing with the same
rate, and so we can apply the central limit theorem for $\alpha$-mixing sequences (see
e.g. Theorem 27.4 in \cite{Bill}) to claim that the distribution of
$n^{1/2}(\hat{p}^0_k-p_k)$ converges weakly to $N (0,v^2_k)$ with
\begin{align*}
v^2_k &= \E (Z_0-p_k)^2  + 2\sum\limits_{l=1}^\infty \E (Z_0-p_k)(Z_l-p_k) \\
    &= p_k(1-p_k)+ 2 \sum\limits_{l=1}^\infty \left(\E Z_0 Z_l -p_k^2\right).
\end{align*}
Since $\E\, Z_0Z_l  = \Pb(X_0>0,X_k>0, X_l>0,X_{l+k}>0)$, this yields
representation~\eqref{variance}.

Now as $g$ is differentiable, $g'\neq 0$, the standard application of the delta-method
(see e.g.\ Theorem~3 in Section~1.5 of~\cite{Borov}) establishes that
\[
\sqrt{n}(g(\hat{p}^0_k)- g(p_k))/k = \sqrt{n}( \hat\alpha^0_k- \alpha)/k
\]
is asymptotically normal with zero mean and variance $v^2_k (g'(p_k))^2 k^{-2}$. The
theorem is proved.
\end{proof}

\begin{proof}[Proof of Theorem~\ref{prop_4}]
That $\hat{p}_k^0$ is an unbiased estimator of $p_k$ is obvious. To compute the
variance of $\hat{p}_k^0$, we set $n_k:=n-k$ and $ \mathbf{1}_i:= \mathbf{1}(X_i>0)$
for brevity and note that
\begin{align}
\E  (\hat{p}_k^0)^2
 &=  n_k^{-2} \E\left(
 \sum_{i=1}^{n_k} \mathbf{1}_i \mathbf{1}_{i+k}
 + 2 \sum_{i=1}^{n_k-1} \sum_{j=i+1}^{n_k} \mathbf{1}_i \mathbf{1}_{i+k} \mathbf{1}_j \mathbf{1}_{j+k}
 \right)
  \notag\\
 & = \frac{1}{4n_k} + \frac{2}{n_k^2}\left(\sum_{i=1}^{n_k-k} \sum_{j=i+1}^{n_k} \cdots + \sum_{i=n_k-k+1}^{n_k-1} \sum_{j=i+1}^{n_k}
 \cdots\right).
  \label{sumsa0}
\end{align}

For each $i$ in the first double sum in the last line the value $j=i+k$ will be
present in the inner sum, and it is clear that the corresponding term will be
equal to
\[
\E \mathbf{1}_i \mathbf{1}_{i+k} \mathbf{1}_{i+k} \mathbf{1}_{i+2k} = 1/8 ,
\]
while all the remaining terms will be equal to 1/16. In the second double sum, one
cannot have $j=i+k$, and so all the terms in it will be 1/16. Therefore
\begin{align}
\E  (\hat{p}_k^0)^2
 &= \frac{1}{4n_k} + \frac{2}{n_k^2}\left(\sum_{i=1}^{n_k-k}
  \left(\frac{n_k-i-1}{16} + \frac1{8}\right) + \sum_{i=n_k-k+1}^{n_k-1}
  \frac{n_k-i}{16}\right)
   \notag\\
  & = \frac{1}{4n_k} + \frac{2}{n_k^2}\left(\sum_{i=1}^{n_k-1} \frac{n_k-i}{16}
    +    \frac{n_k-k}{16}\right)
    \notag\\
    &= \frac1{16} + \frac5{16n_k} - \frac{k}{8n_k^2},
  \label{sumsa}
\end{align}
so that ${\rm Var\,} (\hat{p}_k^0) = \E  (\hat{p}_k^0)^2 - 1/16 = 5/(16n_k) -
k/(8n_k^2)$.

Now we turn to $\hat{p}_k$. Recall that $X_{1,n}\ge \cdots\ge X_{n,n}$ denote
the order statistics for our sample and let ${\mbox{\boldmath$X$}}:= (X_{1,n},
\ldots,X_{n,n})$, ${\mbox{\boldmath$x$}}:= (x_1, \ldots, x_n)\in\R^n$. Assume
for definiteness that $n$ is odd and set $\nu:=n/2+1/2$, so that $\hat m_n (X)=
X_{\nu,n}$. Then by symmetry and the total probability formula  one has
\begin{align*}
\E \hat{p}_k
 & =
 \frac{1}{n_k} \sum_{j=1}^{n_k} \Pb (X_j >\hat m_n (X), X_{j+k}>\hat m_n (X))\\
 & = \Pb (X_1 >\hat m_n (X), X_{2}>\hat m_n (X))
 \\
 & =  \int_{\R^n} \Pb(X_1>x_\nu, X_{2}>x_\nu \mid
{\mbox{\boldmath$X$}}= {\mbox{\boldmath$x$}})\Pb({\mbox{\boldmath$X$}}\in
d{\mbox{\boldmath$x$}}),
\end{align*}
where, again by virtue of symmetry, for ${\mbox{\boldmath$x$}}$ with $x_1
>x_2 >\cdots> x_n$,
\begin{align}
P_2 :& =
 \Pb(X_1>x_\nu, X_2>x_\nu \mid {\mbox{\boldmath$X$}}= {\mbox{\boldmath$x$}})
 \notag\\
 &= \Pb\left(\{X_1,X_2\}\subset \{X_{1,n},X_{2,n},\ldots, X_{\nu-1,n}\}\right)
  \notag\\
&= \frac{\binom{\nu -1 }{2}}{\binom{n }{2}}=  \frac{n-3}{4n}, \label{factor}
\end{align}
so that
\begin{equation}
 \label{p_k_mean}
\E \hat{p}_k =P_2= 1/4 - 3/(4n). %\frac14 - \frac3{4n}
\end{equation}

To find the second moment of $\hat{p}_k$, we use a representation similar to
~\eqref{sumsa0}, the total probability formula and a calculation similar
to~\eqref{sumsa}  to conclude that (keeping $k$ fixed)
\begin{align}
\E \hat{p}_k^2
  & =  \frac{1}{n_k } P_2 +\frac{2}{n_k^2}\left(\sum_{i=1}^{n_k-1}
 (n_k-i)P_4   +(n_k-k)(P_3-P_4)\right)
 \notag\\
 & = P_4 + n^{-1} (P_2 + 2P_3 - 3P_4) + O(n^{-2}),
 \label{p_k_seco}
\end{align}
where, similarly to~\eqref{factor},
\[
P_3 := \Pb(X_1>x_\nu, X_2>x_\nu , X_3>x_\nu \mid {\mbox{\boldmath$X$}}=
{\mbox{\boldmath$x$}}) = \frac{\binom{\nu -1 }{3}}{\binom{n }{3}}=  \frac{(n-3)(n-5)}{8n(n-2)}
\]
and
\[
P_4 := \Pb(X_1>x_\nu, X_2>x_\nu , X_3>x_\nu , X_4>x_\nu \mid {\mbox{\boldmath$X$}}=
{\mbox{\boldmath$x$}}) = \frac{\binom{\nu -1 }{4}}{\binom{n }{4}}= \frac{(n-5)(n-7)}{16 n(n-2)}.
\]
Now substituting the values of $P_r$ into \eqref{p_k_seco} we obtain that
\[
\E \hat{p}_k^2 = \frac1{16} - \frac{5}{16n} + O(n^{-2}),
\qquad {\rm Var\,} (\hat{p}_k) = \E \hat{p}_k^2- P_2^2= \frac{1}{16n} + O(n^{-2}),
\]
and the quadratic error $\E (\hat{p}_k- p_k)^2$ will have the latter form, too,
since $\E \hat{p}_k- p_k = O(n^{-1})$.

If $n$ is even, setting $\nu:=n/2 +1$ leads to the same expressions for $P_r$
in terms of $\nu$ and $n$, resulting in the values
\[
P_2= \frac{n-2}{4(n-1)}= \frac14 -\frac1{4n}+O(n^{-2}),
 \qquad P_3= \frac{n-4}{8(n-1)},
 \qquad P_4= \frac{(n-4)(n-6)}{16(n-1)(n-3)},
\]
so that
\[
\E \hat{p}_k^2 = \frac1{16} - \frac{1}{16n} + O(n^{-2}),
\qquad {\rm Var\,} (\hat{p}_k) = \E \hat{p}_k^2- P_2^2= \frac{1}{16n} + O(n^{-2}),
\]
as before, which completes the proof.
\end{proof}

\noindent {\bf Acknowledgements.} This research was supported by the ARC Centre of
Excellence for Mathematics and Statistics of Complex Systems (MASCOS).

\end{document}